\documentclass[11pt]{amsart}
\usepackage{hyperref}
\usepackage[centertags]{amsmath}
\usepackage{amsfonts}
\usepackage{amssymb}
\usepackage{amsthm}
\usepackage{amsmath}
\usepackage{booktabs}
\usepackage{cite}
\usepackage{comment}
\usepackage{dsfont}
\usepackage{float}
\usepackage{subfigure}
\usepackage[shortlabels]{enumitem}
\usepackage{etoolbox}
\usepackage{needspace}
\usepackage{xcolor}
\usepackage{tikz}
\usetikzlibrary{matrix,arrows}

\theoremstyle{plain}
\newtheorem{thm}{Theorem}[section]
\newtheorem{cor}[thm]{Corollary}
\newtheorem{lem}[thm]{Lemma}
\newtheorem{prop}[thm]{Proposition}

\theoremstyle{definition}

\AtBeginEnvironment{lem}{\Needspace{3\baselineskip}}
\AtBeginEnvironment{thm}{\Needspace{3\baselineskip}}
\AtBeginEnvironment{prop}{\Needspace{3\baselineskip}}
\AtBeginEnvironment{cor}{\Needspace{3\baselineskip}}

\setlist[enumerate,1]{leftmargin=2.2em}

\def\N{\mathbb N}

\def\D{\mathcal D}
\def\S{\mathfrak S}
\def\Z{\mathbb Z}
\def\F{\mathbb F}

\def\sl_2{\mathfrak{sl}_2}
\def\U{U_q(\mathfrak{sl}_2)}
\def\V{U_q'(\mathfrak{so}_3)}
\def\e{\varepsilon}

\def\AW{{\rm AW}}
\def\T{\mathbf T}
\def\M{\mathbf M}

\title[Irreducible modules of Askey-Wilson algebras]{Finite-dimensional irreducible modules of the universal Askey-Wilson algebra}

\author{Hau-Wen Huang}
\address{
Einstein Institute of Mathematics\\
Hebrew University\\
Jerusalem 91904 Israel
}
\email{hauwenh@math.huji.ac.il}

\thanks{The research was supported by National Center for Theoretical Sciences of Taiwan and the Council for Higher Education of Israel.
}

\begin{document}

%\noindent Dear Editors,

%\bigskip

%\noindent I am enclosing herewith a manuscript entitled ``Finite-dimensional irreducible modules of the universal Askey-Wilson algebra'' for possible publication in Duke Mathematical Journal. The main result of this paper is to classify the finite-dimensional irreducible modules of a universal analog of the Askey-Wilson algebra when $q$ is not a root of unity. This problem was open since the Askey-Wilson algebra was defined in 1991. None of the material in the paper has been published or is under consideration for publication elsewhere.
%\medskip

%\noindent Thank you for considering this work and look forward to your response. Please direct all correspondence about this manuscript to me.

%\bigskip

%\noindent Sincerely,\\
%Hau-Wen Huang\\
%Einstein Institute of Mathematics\\
%Hebrew University of Jerusalem\\
%{\tt hauwenh@math.huji.ac.il}

%\newpage

\maketitle

\begin{abstract}
Since the introduction of Askey-Wilson algebras by Zhedanov in 1991, the classification of the finite-dimensional irreducible modules of Askey-Wilson algebras remains open. A universal analog $\triangle_q$ of the Askey-Wilson algebras was recently studied.
In this paper, we consider a family of infinite-dimensional $\triangle_q$-modules. By the universal property of these $\triangle_q$-modules, we classify the finite-dimensional irreducible $\triangle_q$-modules when $q$ is not a root of unity.
\end{abstract}


{\footnotesize {\bf 2010 MSC Primary:} 33C45, 33D80;
{\bf Secondary:} 17B37}.

%{\footnotesize {\bf 2010 Mathematics Subject Classification:} 17B37, 33C45, 33D80.}

\section{Introduction}\label{s:intro}

In his pioneering work of 1991, Zhedanov introduced the Askey-Wilson algebras \cite{hidden_sym} motivated by the Racah coefficients of $\mathfrak{su}_q(2)$
\cite{gz92} and the hidden relations between the Askey-Wilson operator and the three-term recurrence relation of the Askey-Wilson polynomials \cite{awpoly}. These algebras are associative unital algebras over the complex number field involving a nonzero scalar $q$ and five parameters $\varrho,$ $\varrho^*,$ $\eta,$ $\eta^*$, $\omega$. Given these data, the Askey-Wilson algebra $\AW_q$ is defined by generators $K_0$, $K_1$, $K_2$ subject to the following relations:
\begin{eqnarray*}
qK_1K_2-q^{-1}K_2K_1&=&\omega K_1+\varrho K_0+\eta^*,\\
qK_2K_0-q^{-1}K_0K_2&=&\omega K_0+\varrho^* K_1+\eta,\\
qK_0K_1-q^{-1}K_1K_0&=&K_2.
\end{eqnarray*}
Let us abbreviate $\AW=\AW_q$.
For example, the quantum group $\V$ \cite{fair1990,hkp1999,hp2001,hp2011,odes1986, nrz:1990,nr:1993,ik2001,nuw:1996}
different from the Drinfeld-Jimbo type is the algebra $\AW$ with $\varrho=1$, $\varrho^*=1$, $\eta=0$, $\eta^*=0$, $\omega=0$ and the Bannai-Ito algebra \cite{gvz2013,gvz2014,gvz2015,tvz2012} is the limit case $q\to -1$ of $\AW$.
Over two decades of research, the Askey-Wilson algebras have been found applications to
 the quantum integrable systems \cite{base2005,BDK:1994,glz92},
the Drinfeld-Jimbo quantum group $\U$  \cite{cm2013,gz93,uaw&equit2011}, the double affine Hecke algebras of rank one \cite{Z3&daha1,aw&daha1,aw&daha2}, the sixth Painlev\'{e} equation \cite{maz2013}, the discrete quantum mechanics \cite{dqm10,dqm11} and so on. In this paper we study the representation theory of the Askey-Wilson algebras.

The first family of finite-dimensional $\AW$-modules was constructed in \cite[Section~2]{hidden_sym}. On these $\AW$-modules $K_0$, $K_1$ act like Leonard pairs. Roughly speaking, the Leonard pair is a pair of diagonalizable linear transformations on a nonzero finite-dimensional vector space,
each of which acts in an irreducible tridiagonal fashion on
an eigenbasis of the other one \cite[Definition~1.1]{lp2001}. According to their corresponding orthogonal polynomials, Leonard pairs were classified into the $q$-Racah, Racah and related types  \cite[Section~35]{ter2006}. The works \cite{lp&awrelation,vid07} of Terwilliger and Vidunas gave a more comprehensive description of how Leonard pairs are related to $\AW$-modules: %By \cite[Theorem~1.5]{lp&awrelation} and \cite[Section~5]{vid07},
Given a Leonard pair of $q$-Racah or other $q$-types, the underlying vector space supports an irreducible $\AW$-module with appropriate parameters $\varrho$, $\varrho^*$, $\eta$, $\eta^*$, $\omega$ on which $K_0$, $K_1$ act as affine transformations of the Leonard pair. Conversely, assume that $V$ is a finite-dimensional irreducible $\AW$-module on which each of $K_0$, $K_1$ is diagonalizable with all eigenspaces of dimension one. Then $K_0,$ $K_1$ act on $V$ as a Leonard pair of $q$-type, provided that $q$ is not a root of unity. %\cite[Theorem~6.2]{lp&awrelation}.

The notion of Leonard pairs was extended to so-called Leonard triples by Curtin \cite[Definition~2.1]{cur2007}. The finite-dimensional irreducible $\V$-modules for $q$ not a root of unity were classified by Havl\'i\v{c}ek and Po\v{s}ta \cite[Theorem~4]{hp2001}. Based on their results, it can be shown the action of $K_0$, $K_1$, $K_2$ on each of these irreducible $\V$-modules as a Leonard triple. The irreducible $\V$-modules at $q$ a root of unity were proved to be finite-dimensional and studied deeply in \cite[Sections~5--7]{hp2001}.
However the problem of classifying all finite-dimensional irreducible $\AW$-modules with arbitrary parameters is still open. In a recent paper \cite{uaw2011} of Terwilliger, the universal Askey-Wilson algebra $\triangle=\triangle_q$ with $q^4\not=1$ was introduced.
The algebra $\triangle$ is generated by $A$, $B$, $C$ subject to the relations which assert that each of
\begin{eqnarray}\label{eq:comlist}
 A+ \frac{qBC-q^{-1}CB}{q^2-q^{-2}},
\qquad \quad
B+
\frac{qCA-q^{-1}AC}{q^2-q^{-2}},
\qquad \quad
C+
\frac{qAB-q^{-1}BA}{q^2-q^{-2}}
\end{eqnarray}
commutes with $A$, $B$, $C$.
This algebra $\triangle$ is obtained from $\AW$ by the following two-step procedure. First, the algebra is renormalized by a mild change of generators. Second, the remaining parameters are interpreted as central elements in the algebra.

The purpose of this paper is to classify the finite-dimensional irreducible $\triangle$-modules for $q$ not a root of unity. We begin with an infinite-dimensional $\triangle$-module $M_\lambda(a,b,c)$ with four nonzero parameters $a$, $b$, $c$, $\lambda$ which is regarded as the Verma $\triangle$-module due to its significant universal property. Fix an integer $n\geq 0$. An $(n+1)$-dimensional $\triangle$-module $V_n(a,b,c)$ is explicitly constructed by taking a quotient of $M_\lambda(a,b,c)$ with $\lambda=q^n$. The irreducibility criterion for $V_n(a,b,c)$ can be simply characterized as
\begin{gather*}
abc,\, a^{-1}bc,\, ab^{-1}c,\, abc^{-1}\notin\{q^{1-n},q^{3-n},\ldots,q^{n-1}\}.
\end{gather*}
Consider the set $\T$ consisting of all such triples $(a,b,c)$. There is an action of the group $\{\pm1\}^3$ on $\T$ given by
\begin{equation*}
(a,b,c)^{(-1,1,1)}=(a^{-1},b,c), \qquad \quad
(a,b,c)^{(1,-1,1)}=(a,b^{-1},c), \qquad \quad
(a,b,c)^{(1,1,-1)}=(a,b,c^{-1})
\end{equation*}
for all $(a,b,c)\in \T$. Let $\T/\{\pm 1\}^3$ denote the set of all $\{\pm1\}^3$-orbits of $\T$. For $(a,b,c)\in \T$ let $[a,b,c]$ denote the $\{\pm 1\}^3$-orbit of $\T$ that contains $(a,b,c)$. Define $\M$ to be the set of the isomorphism classes of irreducible $\triangle$-modules that have dimension $n+1$. By the universal property of Verma $\triangle$-modules we establish a bijection $\T/\{\pm1\}^3\to \M$ given by
\begin{equation*}
%\begin{array}{ccc}
%\T/\{\pm1\}^3
%\qquad &\to
%\qquad &\M\\
[a,b,c] \quad \mapsto \quad \hbox{the isomorphism class of $V_n(a,b,c)$} \qquad \quad \hbox{for all $[a,b,c]\in \T/\{\pm1\}^3$}.
%\end{array}
\end{equation*}
This result gives a classification of the finite-dimensional irreducible $\triangle$-modules when $q$ is not a root of unity.

Besides, we characterize on which $\triangle$-modules $A$, $B$, $C$ give Leonard pairs or a Leonard triple and formulate the sufficient conditions for $\triangle$-modules to be unitary.
Apply our classification to $\V$ and compare the result with \cite[Theorem~4]{hp2001}. Determine how many $\U$-modules on $V_n(a,b,c)$ give the $\triangle$-module $V_n(a,b,c)$ by pulling back via the homomorphism $\triangle\to \U$ given below \cite[Proposition~1.1]{uaw&equit2011}. Close the paper with an illustration how the Racah coefficients of $\U$ are related to the $\triangle$-modules.

\section{Notation and preliminaries}\label{s:pre}

Before launching into the subject we lay some groundwork in preparation. Because our arguments are valid for any algebraically closed field $\F$, we change the underlying field from the complex number field to $\F$. Let $\Lambda$, $Q$, $X$, $Y$, $Z$ denote five mutually commuting indeterminates over $\F$.  Let $\Z$ denote the ring of integers. Let $\N$ denote the set of the nonnegative integers and $\N^*=\N\setminus\{0\}$. Define
\begin{eqnarray*}
\theta_i(\Lambda,Q;X) &=&\Lambda Q^{-2i} X^{-1}+\Lambda^{-1}Q^{2i}X
\qquad \quad \hbox{for $i\in \Z$},\\
\phi_i(\Lambda,Q;X,Y,Z) &=&
\Lambda Q X^{-1}Y^{-1} (Q^i-Q^{-i})
(\Lambda^{-1}Q^{i-1}-\Lambda Q^{1-i}) \\
&& \quad \times\;\; (Q^{-i}-\Lambda^{-1}Q^{i-1}XYZ)
(Q^{-i}-\Lambda^{-1}Q^{i-1}XYZ^{-1})
\qquad \quad \hbox{for $i\in \Z$},\\
\omega(\Lambda,Q;X,Y,Z) &=&
(\Lambda Q+ \Lambda^{-1} Q^{-1})(Z+Z^{-1})
+(X+X^{-1})(Y+Y^{-1}).
\end{eqnarray*}
Observe that
\begin{eqnarray}
\phi_i(\Lambda,Q;X,Y,Z) &\in& \F[\Lambda,Q,X,Y,Z+Z^{-1}] \qquad \quad \hbox{for $i\in \Z$}, \label{e:phi}\\
\omega(\Lambda,Q;X,Y,Z) &\in& \F[\Lambda Q+ \Lambda^{-1} Q^{-1},X+X^{-1},Y+Y^{-1},Z+Z^{-1}] \label{e:omega}
\end{eqnarray}
and
\begin{eqnarray}
\theta_i(\Lambda^{-1},Q^{-1};X^{-1}) &=& \theta_i(\Lambda,Q;X) \qquad \quad \hbox{for $i\in \Z$},\label{e:theta-1}\\
\phi_i(\Lambda^{-1},Q^{-1};X^{-1},Y^{-1},Z^{-1})
&=& \phi_i(\Lambda,Q;X,Y,Z) \qquad \quad \hbox{for $i\in \Z$},
\label{e:phi-1}\\
%\omega(\Lambda^{-1},Q^{-1};X^{-1},Y^{-1},Z^{-1}) &=&\omega(\Lambda,Q;X,Y,Z),\label{e:omega-1}\\
\omega(\Lambda,Q;Y,X,Z) &=&\omega(\Lambda,Q;X,Y,Z). \label{e:omega-2}
\end{eqnarray}

Fix $n\in \N$ throughout this paper. Let
$$
\{\theta_i(Q;X)\}_{i\in \Z}, \qquad \quad
\{\phi_i(Q;X,Y,Z)\}_{i\in \Z}, \qquad \quad
\omega(Q;X,Y,Z)
$$
denote the Laurent polynomials $\{\theta_i(\Lambda,Q;X)\}_{i\in \Z}$, $\{\phi_i(\Lambda,Q;X,Y,Z)\}_{i\in \Z}$, $\omega(\Lambda,Q;X,Y,Z)$ with the substitution $\Lambda=Q^n$, respectively. Observe that
\begin{eqnarray}
\theta_{n-i}(Q;X)&=&\theta_i(Q^{-1};X) \qquad \quad \hbox{for $i\in \Z$},\label{e:theta-2}\\
\phi_{n-i+1}(Q;X,Y,Z) &=&\phi_i(Q^{-1};X,Y,Z)
\qquad \quad \hbox{for $i\in \Z$}. \label{e:phi-2}
\end{eqnarray}
Note that (\ref{e:phi})--(\ref{e:phi-2}) will be used without further mention.

\subsection{Three $\N\times \N$ matrices}

A (possibly infinite) square matrix is said to be {\it tridiagonal} if each nonzero entry lies on either the diagonal, the subdiagonal, or the superdiagonal. A square matrix is said to be {\it lower} (resp.\! {\it upper}) {\it bidiagonal} if each nonzero entry lies on the diagonal or subdiagonal (resp. superdiagonal).

The $\N\times \N$ matrices $L(\Lambda,Q;X)$, $U(\Lambda,Q;X,Y,Z)$, $T(\Lambda,Q;X,Y,Z)$ given below will be used to define the Verma $\triangle$-modules. The matrix $L(\Lambda,Q;X)$ is lower bidiagonal with
\begin{eqnarray*}
L(\Lambda,Q;X)_{ii} &=& \theta_i(\Lambda,Q;X) \qquad \quad \hbox{for $i\in \N$},\\
L(\Lambda,Q;X)_{i,i-1} &=& 1 \qquad \quad \hbox{for $i\in \N^*$}.
\end{eqnarray*}
The matrix $U(\Lambda,Q;X,Y,Z)$ is upper bidiagonal with
\begin{eqnarray*}
U(\Lambda,Q;X,Y,Z)_{ii} &=& \theta_i(\Lambda,Q;Y) \qquad \quad \hbox{for $i\in \N$},\\
U(\Lambda,Q;X,Y,Z)_{i-1,i} &=& \phi_i(\Lambda,Q;X,Y,Z)  \qquad \quad \hbox{for $i\in \N^*$}.
\end{eqnarray*}
The matrix $T(\Lambda,Q;X,Y,Z)$ is tridiagonal with
\begin{eqnarray*}
T(\Lambda,Q;X,Y,Z)_{ii} &=&
\frac{Q^{-1}\phi_{i+1}(\Lambda,Q;X,Y,Z)-Q\,\phi_i(\Lambda,Q;X,Y,Z)}{Q^2-Q^{-2}}
\\
%-\frac{\theta_i(\Lambda,Q;X) \theta_i(\Lambda,Q;Y)}{Q+Q^{-1}}\\
&&\quad+\;\;
\frac{\omega(\Lambda,Q;X,Y,Z)-\theta_i(\Lambda,Q;X) \theta_i(\Lambda,Q;Y)}{Q+Q^{-1}}
\qquad \quad \hbox{for $i\in \N$},\\
T(\Lambda,Q;X,Y,Z)_{i-1,i} &=& \frac{Q^{-1}\theta_i(\Lambda,Q;X)-Q\,\theta_{i-1}(\Lambda,Q;X)}{Q^{2}-Q^{-2}}\,\phi_i(\Lambda,Q;X,Y,Z)  \qquad \quad \hbox{for $i\in \N^*$},\\
T(\Lambda,Q;X,Y,Z)_{i,i-1} &=& \frac{Q^{-1}\theta_i(\Lambda,Q;Y)-Q\,\theta_{i-1}(\Lambda,Q;Y)}{Q^2-Q^{-2}} \qquad \quad \hbox{for $i\in \N^*$}.
\end{eqnarray*}

The $\N\times \N$ matrices $E(\Lambda,Q;X)$, $S(\Lambda,Q;X,Y,Z)$ from two matrix equations will also be present in this paper. To solve the two equations we transform either of them into a recurrence relation. After straightforward observations and tedious verifications, the solutions are available below.
Recall the notation
$$
[i]_Q =\frac{Q^{\,i}-Q^{-i}}{Q-Q^{-1}} \qquad \quad \hbox{for $i\in \Z$}
$$
and the Gaussian binomial coefficients
$$
{j\brack i}_{\! Q}=\prod_{h=1}^i\frac{[j-h+1]_Q}{[h]_Q}
\qquad \quad\hbox{for $i\in \N$ and $j\in \Z$}.
$$

\begin{lem}\label{lem:recrel1}
\begin{enumerate}
\item  The upper triangular $\N\times \N$ matrix $E(\Lambda,Q;X)$ with $(i,j)$-entry
\begin{gather*}
%E(\Lambda,Q;X)_{ij} =
{j\brack i}_{\!Q}\prod_{h=1}^{j-i}
(\Lambda^{-1}Q^{h+i-1}-\Lambda Q^{1-i-h})
(Q^{1-h}X-Q^{h-1}X^{-1})
\end{gather*}
for all $i,j\in \N$ with $i\leq j$ is the unique
matrix $E$ satisfying
$$
L(\Lambda,Q;X)\cdot E=E\cdot L(\Lambda,Q;X^{-1})
$$
with $E_{00} = 1$ and $E_{i0} = 0$ for all $i\in \N^*$.

\item The upper triangular $\N\times \N$ matrix $S(\Lambda,Q;X,Y,Z)$ with $(i,j)$-entry
\begin{gather*}
(-1)^i
\Lambda^i
Q^{-ij} Y^{-j}
{j \brack i}_{\!Q}
\displaystyle\prod_{h=1}^{j-i}
(\Lambda Q^{h-j} - \Lambda^{-1} Q^{j-h})
(Q^{j-h}Y Z- \Lambda Q^{h-j-1}X^{-1})
\end{gather*}
for all $i,j\in \N$ with $i\leq j$ is the unique matrix $S$ satisfying
$$
T(\Lambda,Q;X,Y,Z)\cdot S=S\cdot L(\Lambda,Q;Z)
$$
with $S_{00} = 1$ and $S_{i0} = 0$ for all $i\in \N^*$.
\end{enumerate}
\end{lem}

While studying the finite-dimensional irreducible $\triangle$-modules, we will turn our attention to the submatrices
%$L(Q;X)$, $U(Q;X,Y,Z)$, $T(Q;X,Y,Z)$, $E(Q;X)$, $S(Q;X,Y,Z)$
\begin{gather*}
L(Q;X), \qquad \quad
U(Q;X,Y,Z), \qquad \quad
T(Q;X,Y,Z)
\end{gather*}
of $L(\Lambda,Q;X)$, $U(\Lambda,Q;X,Y,Z)$, $T(\Lambda,Q;X,Y,Z)$ indexed by the first $n+1$ rows and $n+1$ columns with the substitution $\Lambda=Q^n$, respectively. Also, we will see the $(n+1)\times (n+1)$ matrices $E(Q;X)$, $S(Q;X,Y,Z)$ defined in the same way and the matrices $F(Q;X,Y,Z)$, $P(Q;X,Y,Z)$ solved from the following matrix equations.

\begin{lem}\label{lem:recrel2}
\begin{enumerate}
\item The lower triangular $(n+1)\times (n+1)$ matrix $F(Q;X,Y,Z)$ with $(i,j)$-entry
\begin{gather*}
%F(Q;X,Y,Z)_{ij} &=
{i\brack j}_{\!Q}
\prod_{h=1}^{i-j} (\theta_0(Q;Y)-\theta_{h-1}(Q^{-1};Y))
\prod_{h=1}^j
\frac{[n-i+h]_Q}{[n-h+1]_Q}\,
\phi_h(Q;X,Y,Z)
\prod_{h=1}^{n-i} \phi_h(Q;X^{-1},Y,Z)
%\qquad \quad (0\leq j\leq i\leq n)
\end{gather*}
for all $0\leq j\leq i\leq n$ is the unique matrix $F$ satisfying
$$
L(Q;X)\cdot F=F\cdot L(Q;X)
$$
with
$
F_{i0}=\prod\limits_{h=1}^{i}(\theta_0(Q;Y)-\theta_{h-1}(Q^{-1};Y))
\prod\limits_{h=1}^{n-i}\phi_h(Q;X^{-1},Y,Z)
$ for all $0\leq i\leq n$.

\item The upper left triangular $(n+1)\times (n+1)$ matrix $P(Q;X,Y,Z)$ with $(n-i,j)$-entry
\begin{gather*}
{i\brack j}_{\!Q}
\prod_{h=1}^{i-j} (\theta_0(Q;X)-\theta_{h-1}(Q^{-1};X))
\prod_{h=1}^j\frac{[n-i+h]_Q}{[n-h+1]_Q}\,\phi_h(Q;X,Y^{-1},Z)
\end{gather*}
for all $0\leq j\leq i\leq n$ is the unique
matrix $P$ satisfying
$$
U(Q;X,Y,Z)\cdot P=P\cdot L(Q^{-1};Y)
$$
with
$
P_{n-i,0}=\prod\limits_{h=1}^{i}(\theta_0(Q;X)-\theta_{h-1}(Q^{-1};X))$ for all $0\leq i\leq n$.
\end{enumerate}
\end{lem}

\subsection{The universal Askey-Wilson algebra}

%Fix a nonzero scalar $q\in \F$ with $q^4\not=1$.
Define $\alpha$, $\beta$, $\gamma$ to be the central elements of $\triangle$ obtained from multiplying the elements in (\ref{eq:comlist}) by $q+q^{-1}$, respectively. In terms of $A,$ $B,$ $\gamma$
\begin{eqnarray}
C &=& \frac{\gamma}{q+q^{-1}}-\frac{q A B-q^{-1} B A}{q^2-q^{-2}},
\label{C->ABc} \\
\alpha &=& \frac{B^2A-(q^2+q^{-2})BAB+AB^2+(q^2-q^{-2})^2A+(q-q^{-1})^2B\gamma}{(q-q^{-1})(q^2-q^{-2})}, \label{a->ABc}\\
\beta &=& \frac{A^2B-(q^2+q^{-2})ABA+BA^2+(q^2-q^{-2})^2B+(q-q^{-1})^2A\gamma}{(q-q^{-1})(q^2-q^{-2})}. \label{b->ABc}
\end{eqnarray}
By (\ref{C->ABc}) the elements $A$, $B$, $\gamma$ form a set of generators of $\triangle$. The Poincar\'{e}-Birkhoff-Witt theorem for $\triangle$ was proved in \cite[Theorem~4.1]{uaw2011}:

\begin{lem}\label{Delta:basis}
%{\rm\cite[Theorem~4.1]{uaw2011}.}
The monomials
$A^i C^j B^k \alpha^r \beta^s \gamma^t$ for all $i,j,k,r,s,t\in \N$ form a basis of $\triangle$.
\end{lem}

Recall that the symmetry group $\S_3$ of degree three has a presentation with generators $\sigma$, $\tau$ and relations
$\sigma^2=1$, $\tau^2=1$, $(\sigma\tau)^3=1$.
Define an action of $\S_3$ on the set of
all permutations $(R,S,T)$ of $A$, $B$, $C$ by
\begin{gather*}
(R,S,T)^\sigma=(T,S,R), \qquad \quad (R,S,T)^\tau=(S,R,T).
\end{gather*}
Similarly define an action of $\S_3$ on the set of all permutations $(r,s,t)$ of $\alpha$, $\beta$, $\gamma$ by
\begin{gather*}
(r,s,t)^\sigma=(t,s,r), \qquad \quad (r,s,t)^\tau=(s,r,t).
%(A^\sigma,B^\sigma,C^\sigma) = (C,B,A),\qquad \quad
%(A^\tau,B^\tau,C^\tau) = (B,A,C),
\end{gather*}
Let
$$
C^\vee=C+\frac{AB-BA}{q-q^{-1}}.
$$
As a consequence of \cite[Theorem~3.1]{uaw2011} we have
%It is routine to verify that

\begin{lem}\label{lem:auto}
For each $g\in \S_3$ there exists a unique automorphism $\widetilde g:\triangle\to \triangle$ that sends
\begin{gather*}
%(A^g,B^g,\alpha^g,\beta^g,\gamma^g) \quad \mapsto \quad
%(A,B,\alpha,\beta,\gamma),
%\begin{array}{lll}
(A,B,C)^g
\quad \mapsto \quad
\left\{
\begin{array}{ll}
(A,B,C) \qquad \quad &\hbox{if the permutation $g$ is even},\\
(A,B,C^\vee) \qquad  \quad &\hbox{if the permutation $g$ is odd},
%\end{array}
%\right.\\[8pt]
%\alpha^g \;  \mapsto \; \alpha,
%&\beta^g \; \,\mapsto \; \, \beta,
%&\gamma^g\; \,\mapsto \; \, \gamma.
\end{array}
\right.
\end{gather*}
and sends $(\alpha,\beta,\gamma)^g
\;\mapsto \;
(\alpha,\beta,\gamma)$.
\end{lem}

\noindent Given a $\triangle$-module $V$ and $g\in \S_3$ the notation $V^g$ will stand for the $\triangle$-module obtained by pulling back $V$ via $\widetilde g$. %Given a finite-dimensional $\triangle$-module $V$, if there is no confusion the notation ``${\rm tr}$'' will stand for the trace map of $\triangle$ on $V$.

\section{The Verma $\triangle$-module $M_\lambda(a,b,c)$}

Let $a$, $b$, $c$, $\lambda$ denote any four nonzero scalars in $\F$. There is a $\triangle$-module $M_\lambda(a,b,c)$ that has a basis $\{m_i\}_{i\in \N}$ with respect to which the matrices representing $A$, $B$, $C$ are
$$
L(\lambda,q;a), \qquad \quad
U(\lambda,q;a,b,c),\qquad \quad
T(\lambda,q;a,b,c)
$$
respectively. To confirm the existence of this module, one can check that the above matrices satisfy the defining relations for $\triangle$. The central elements $\alpha,$ $\beta,$ $\gamma$ act on $M_\lambda(a,b,c)$ as scalar multiplications by
$$
\omega(\lambda,q;b,c,a),
\qquad \quad
\omega(\lambda,q;c,a,b),
\qquad \quad
\omega(\lambda,q;a,b,c)
$$
respectively. The matrix representing $C^\vee$ with respect to $\{m_i\}_{i\in \N}$ is equal to
\begin{gather}\label{e:Cvee}
T(\lambda^{-1},q^{-1};a^{-1},b^{-1},c^{-1}).
\end{gather}

For a semisimple Lie algebra $\mathfrak g$ over an algebraically closed field of characteristic $0$, by applying the universal property of Verma $\mathfrak g$-modules, every finite-dimensional irreducible $\mathfrak g$-module $V$ is shown as a quotient of the Verma $\mathfrak g$-module with the same highest weight as $V$. The $\triangle$-module $M_\lambda(a,b,c)$ plays a role as the {\it Verma $\triangle$-module} from the above viewpoint. The universal property of $M_\lambda(a,b,c)$ is presented in Section~\ref{s:universal}. Applying the universal property of $M_\lambda(a,b,c)$, in the proof of Theorem~\ref{thm:class} we shall see that every finite-dimensional irreducible $\triangle$-module is a quotient of $M_\lambda(a,b,c)$ with appropriate $a$, $b$, $c$, $\lambda$ when $q$ is not a root of unity.

A connection between the Askey-Wilson polynomials and $M_\lambda(a,b,c)$ is displayed in Section~\ref{s:AW}. The four bases of $M_\lambda(a,b,c)$ mentioned in Section~\ref{s:4} are to prepare for Section~\ref{section:L(a,b,c)}.

\subsection{The universal property of $M_\lambda(a,b,c)$}\label{s:universal}

%In this section we present the universal property of $M_\lambda(a,b,c)$.
%There is another description of $M_\lambda(a,b,c)$.
Let $I_\lambda(a,b,c)$ denote the left ideal of $\triangle$ generated by
\begin{gather}
B-\theta_0(\lambda,q;b), \label{I1}\\
(B-\theta_1(\lambda,q;b)) (A-\theta_0(\lambda,q;a))-
\phi_1(\lambda,q;a,b,c), \label{I2}\\
\alpha-\omega(\lambda,q;b,c,a),\qquad \beta-\omega(\lambda,q;c,a,b),\qquad \gamma-\omega(\lambda,q;a,b,c).  \label{I3}
\end{gather}

\begin{lem}\label{D/K:span}
The $\F$-vector space $\triangle/I_\lambda(a,b,c)$ is spanned by $A^i+I_\lambda(a,b,c)$ for all $i\in \N$.
\end{lem}
\begin{proof}
By Lemma~\ref{Delta:basis} the cosets $A^i C^j B^k \alpha^r \beta^s \gamma^t+I_\lambda(a,b,c)$ for all $i,j,k,r,s,t\in \N$ span $\triangle/I_\lambda(a,b,c)$. By (\ref{I1}) and (\ref{I3}), $A^i C^j B^k \alpha^r \beta^s \gamma^t+I_\lambda(a,b,c)\in \F A^i C^j+I_\lambda(a,b,c)$ for all $i,j,k,r,s,t\in \N$. Therefore it suffices to show that for any $i,j\in \N$ the coset $A^i C^j+I_\lambda(a,b,c)$ is contained in the $\F$-vector space spanned by $A^h+I_\lambda(a,b,c)$ for all $h\in \N$.

To see this we proceed by induction on $j$. There is nothing to prove for $j=0$. Assume that $j\geq 1$. By (\ref{C->ABc}), $A^i C^j$ is a linear combination of $A^iC^{j-1}AB$, $A^iC^{j-1}\gamma$ and $A^i C^{j-1}BA$. By (\ref{I1}), $A^iC^{j-1}AB+I_\lambda(a,b,c)\in \F A^iC^{j-1}A+I_\lambda(a,b,c)$. By (\ref{I3}),  $A^iC^{j-1}\gamma+I_\lambda(a,b,c)\in \F A^iC^{j-1}+I_\lambda(a,b,c)$.
By (\ref{I1}) and (\ref{I2}),  $A^i C^{j-1}BA+I_\lambda(a,b,c)$ is a linear combination of $A^iC^{j-1}+I_\lambda(a,b,c)$ and $A^iC^{j-1}A+I_\lambda(a,b,c)$.
Combining the above comments $A^iC^{j}+I_\lambda(a,b,c)$ is a linear combination of $A^iC^{j-1}+I_\lambda(a,b,c)$ and $A^iC^{j-1}A+I_\lambda(a,b,c)$. The result now follows by induction hypothesis.
\end{proof}

\begin{prop}\label{prop:VermaU}
There exists a unique $\triangle$-module isomorphism $\triangle/I_\lambda(a,b,c)\to M_\lambda(a,b,c)$ that sends $1+I_\lambda(a,b,c)$ to $m_0$.
\end{prop}
\begin{proof}
By construction there exists a unique $\triangle$-module homomorphism $\iota:\triangle/I_\lambda(a,b,c)\to M_\lambda(a,b,c)$ that sends $1+I_\lambda(a,b,c)\mapsto m_0$. The homomorphism $\iota$ sends
\begin{gather}\label{e:coset}
\prod^{i}_{h=1}(A-\theta_{h-1}(\lambda,q;a))+I_\lambda(a,b,c)
\end{gather}
to $m_i$ for all $i\in \N$.
Therefore the cosets (\ref{e:coset}) for all $i\in \N$ are linearly independent over $\F$. By Lemma~\ref{D/K:span} these cosets span $\triangle/I_\lambda(a,b,c)$ and thus form a basis of $\triangle/I_\lambda(a,b,c)$. Therefore $\iota$ is an isomorphism.
\end{proof}

By Proposition~\ref{prop:VermaU}
the Verma $\triangle$-module $M_\lambda(a,b,c)$ has the following universal property:
%\smallskip
%\noindent
If $V$ is a $\triangle$-module and a vector $v\in V$ with
\begin{gather*}
Bv=\theta_0(\lambda,q;b)v, \\
(B-\theta_1(\lambda,q;b)) (A-\theta_0(\lambda,q;a))v=
\phi_1(\lambda,q;a,b,c) v, \\
\alpha v=\omega(\lambda,q;b,c,a) v,\qquad
\beta v=\omega(\lambda,q;c,a,b) v,\qquad
\gamma v=\omega(\lambda,q;a,b,c) v,
\end{gather*}
then there exists a unique $\triangle$-module homomorphism $M_\lambda(a,b,c)\to V$ that sends $m_0$ to $v$.

\subsection{The Askey-Wilson polynomials and $M_\lambda(a,b,c)$}\label{s:AW}

In this section we display that $M_\lambda(a,b,c)$ is isomorphic to the $\triangle$-module constructed from the hidden symmetry of the Askey-Wilson polynomials, provided that $q$ is not a root of unity and
\begin{eqnarray*}
q^{-2i}\not=\lambda^{-2},\;
\lambda^{-2} q^{-2} b^2, \;
\lambda^{-1} q a b c,\;
\lambda^{-1} q a b c^{-1}
\qquad \quad \hbox{for all $i\in \N$}.
%\lambda^{-1}q^{2i-1}b\not=\lambda q^{-1} b,\;
%\lambda q b^{-1}, \;
%q^{-2}a^{-1}c^{-1},\;
%q^{-2}a^{-1}c
%\qquad \quad \hbox{for all $i\in \N$}.
\end{eqnarray*}
Here we identify $X=Y+Y^{-1}$. For each $i\in \N$ let \begin{eqnarray*}
p_i(X)&=&\sum_{j=0}^i\prod\limits_{h=1}^{j} \frac{(\theta_i(\lambda,q;b)-\theta_{h-1}(\lambda,q;b))(X-\theta_{h-1}(\lambda,q;a))} {\phi_h(\lambda,q;a,b,c)}\\
&=&{}_4\phi_3 \Biggl({ {q^{-2i}, \;\; \lambda^{-2}q^{2i}b^2,\;\; \lambda^{-1}a Y,\;\; \lambda^{-1}a Y^{-1}}\atop {\lambda^{-1}q abc,\;\; \lambda^{-1}q abc^{-1},\;\; \lambda^{-2}}}\;\Bigg\vert \; q^2,\;q^2\Biggr),
\end{eqnarray*}
where ${}_4\phi_3$ is the basic hypergeometric series. Recall from \cite{awpoly,Koe2010} that $\{p_i(X)\}_{i\in \N}$ are the Askey-Wilson polynomials with the following properties. The three-term recurrence relation for $\{p_i(X)\}_{i\in \N}$ is
\begin{gather*}
Xp_i(X)=a_i\hspace{0.5mm} p_{i+1}(X) +
c_i\hspace{0.5mm} p_i(X)+
b_i\hspace{0.5mm} p_{i-1}(X)
\qquad \quad \hbox{for $i\in \N$,}
\end{gather*}
where $p_{-1}(X)$ is interpreted to be $0$ and
\begin{eqnarray*}
a_i &=&
\frac{
(\lambda^{-1}q^ib-\lambda q^{-i}b^{-1})
\,\phi_{i+1}(\lambda,q;a,b,c)
}
{
(q^{i+1}-q^{-1-i})
(\lambda^{-1}q^{2i}b-\lambda q^{-2i}b^{-1})
(\lambda^{-1}q^{2i+1}b-\lambda q^{-1-2i}b^{-1})
},\\
b_{i} &=&
\frac
{
(q^ib-q^{-i}b^{-1})
\, \phi_i(\lambda,q;a^{-1},b,c)
}
{
(\lambda^{-1}q^{i-1}-\lambda q^{1-i})
(\lambda^{-1} q^{2i-1}b-\lambda q^{1-2i}b^{-1})
(\lambda^{-1}q^{2i}b-\lambda q^{-2i}b^{-1})
},\\
%\frac{
%a(1-q^{2i})
%(1-q^{2i} b^2)(1-\lambda^{-1}q^{2i-1}a^{-1}bc)
%(1-\lambda^{-1}q^{2i-1}a^{-1}bc^{-1})
%}
%{
%\lambda(1- \lambda^{-2}q^{4i-2} b^2 )
%(1-\lambda^{-2}q^{4i} b^2 )}
%\qquad \quad \hbox{for $i\in \N$,}\\
c_i &=& \theta_0(\lambda,q;a)-a_i-b_i
\end{eqnarray*}
for $i\in \N$.
The Askey-Wilson operator $\D$ is the linear transformation $\F(Y)\to \F(Y)$ defined by
$$
\D\hspace{0.05mm}f(Y) =
A(Y)f(q^2 Y)-(A(Y)+A(Y^{-1})-\lambda b^{-1}-\lambda^{-1}b)f(Y)+A(Y^{-1}) f(q^{-2} Y)
$$
for $f(Y)\in \F(Y)$, where
$$
A(Y)=
\frac{ \lambda(1-\lambda^{-1} a Y) (1-\lambda^{-1} a^{-1} Y)
(1-q bc Y) (1-q bc^{-1} Y) }
{b\,(1-Y^2) (1-q^2 Y^2) }.
$$
For each $i\in \N$ the polynomial $p_i(X)$ is an eigenfunction of $\D$ with respect to the eigenvalue $ \theta_i(\lambda,q;b)$. By \cite[Introduction]{hidden_sym} the polynomial ring $\F[X]$ supports a $\triangle$-module on which
\begin{eqnarray*}
\begin{array}{clclcl}
& &\F[X] \quad  &\to \quad
&\F[X]\\
A &: \quad  &p(X) \quad
&\mapsto \quad & X p(X),\\
%\qquad \quad &\hbox{for all $p(X)\in \F[X]$},\\
B &:\quad   &p(X) \quad
&\mapsto \quad &\D\hspace{0.05mm} p(X)
%\qquad \quad &\hbox{for all $p(X)\in \F[X]$},
\end{array}
\end{eqnarray*}
for all $p(X)\in \F[X]$ and $\gamma$ acts as scalar multiplication by $\omega(\lambda,q;a,b,c)$. By the universal property of $M_\lambda(a,b,c)$ it is routine to show that there is a $\triangle$-module isomorphism
\begin{equation*}
\begin{array}{cccc}
M_\lambda(a,b,c)
\quad  &\to
\quad &\F[X]\\
m_i \quad &\mapsto \quad  &\displaystyle\prod\limits_{h=1}^{i} (X-\theta_{h-1}(\lambda,q;a))
&\qquad \quad \hbox{for all $i\in \N$}.
\end{array}
\end{equation*}

\subsection{Four bases of $M_\lambda(a,b,c)$}\label{s:4}

For convenience the basis $\{m_i\}_{i\in \N}$ of $M_\lambda(a,b,c)$ will be said to be {\it canonical}. In this section we make use of the universal property of the Verma $\triangle$-modules to obtain four bases of $M_\lambda(a,b,c)$ on which the action of $A$, $B$, $C$ is similar to that on the canonical basis. To present the symmetry of these bases,
we define an action of the Klein group $V_4=\{\pm 1\}\times \{1,\sigma\}$ on the set consisting of all $5$-tuples $(\lambda,q^\e;a,b,c)\in \F^5$ with $a,b,c,\lambda\not=0$ and $\e\in\{\pm 1\}$ by
\begin{gather*}
(\lambda,q^\e;a,b,c)^{(-1,1)} = (\lambda,q^\e;a^{-1},b,c^{-1}),\qquad \quad
(\lambda,q^\e;a,b,c)^{(1,\sigma)} = (\lambda^{-1},q^{-\e};c^{-1},b^{-1},a^{-1}).
\end{gather*}
%for all nonzero scalars $a,b,c,\lambda\in \F$ and $\e\in\{\pm1\}$.

\begin{prop}\label{lem:M4bases}
For each $(\e,g)\in V_4$ there exists a unique basis $\{m_i^{(\e,g)}\}_{i\in \N}$ of $M_\lambda(a,b,c)$ with $m_0^{(\e,g)}=m_0$ and with respect to which the matrices representing $(A,B,C)^g$ are
$$
L(\Lambda,Q;X), \qquad \quad
U(\Lambda,Q;X,Y,Z), \qquad \quad
T(\Lambda,Q;X,Y,Z) \qquad \quad
$$
with $(\Lambda,Q;X,Y,Z)=(\lambda,q;a,b,c)^{(\e,g)}$, respectively.
\end{prop}
\begin{proof}
(Uniqueness): From the matrix $L(\lambda,q;a,b,c)$ representing $A$ with respect to the basis $\{m_i^{(1,1)}\}_{i\in \N}$ of $M_\lambda(a,b,c)$, we see that
$$
m_i^{(1,1)}=\prod\limits_{h=1}^{i} (A-\theta_{h-1}(\lambda,q;a))
m_0^{(1,1)}
\qquad \quad \hbox{for all $i\in\N^*$}.
$$
Since $m_0^{(1,1)}$ is fixed these vectors $m_i^{(1,1)}$ for all $i\in \N^*$ are uniquely determined. Therefore the uniqueness of $\{m_i^{(1,1)}\}_{i\in \N}$ follows. By similar arguments the uniqueness of $\{m_i^{(\e,g)}\}_{i\in \N}$ follows for each $(\e,g)\in V_4$.

(Existence): The canonical basis of $M_\lambda(a,b,c)$ is exactly $\{m_i^{(1,1)}\}_{i\in \N}$. Since $V_4$ is generated by $(-1,1)$ and $(1,\sigma)$, it is enough to show the existence of $\{m_i^{(-1,1)}\}_{i\in \N}$ and $\{m_i^{(1,\sigma)}\}_{i\in \N}$. We first prove the existence of  $\{m_i^{(-1,1)}\}_{i\in \N}$. It is routine to verify that
\begin{gather*}
(B-\theta_1(\lambda,q;b))(A-\theta_0(\lambda,q;a^{-1})) m_0
=\phi_1(\lambda,q;a^{-1},b,c^{-1}) m_0.
\end{gather*}
The central elements $\alpha$, $\beta$, $\gamma$ act on $M_\lambda(a,b,c)$ as scalar multiplications by
$$
\omega(\lambda,q;b,c^{-1},a^{-1}), \qquad \quad \omega(\lambda,q;c^{-1},a^{-1},b), \qquad \quad
\omega(\lambda,q;a^{-1},b,c^{-1})
$$
respectively.
Let $\{u_i\}_{i\in \N}$ denote the canonical basis of $M_\lambda(a^{-1},b,c^{-1})$.
By the universal property of $M_\lambda(a^{-1},b,c^{-1})$ there exists a $\triangle$-module isomorphism $M_\lambda(a^{-1},b,c^{-1})\to M_\lambda(a,b,c)$ that sends $u_i$ to
\begin{gather}\label{e:basis1}
%u_i \quad  \mapsto \quad
\prod_{h=1}^{i} (A-\theta_{h-1}(\lambda,q;a^{-1}))m_0 \qquad \quad \hbox{for all $i\in \N$.}
\end{gather}
The matrices representing $A$, $B$, $C$ with respect to the basis $\{u_i\}_{i\in \N}$ of $M_\lambda(a^{-1},b,c^{-1})$ are
$L(\lambda,q;a^{-1})$, $U(\lambda,q;a^{-1},b,c^{-1})$, $T(\lambda,q;a^{-1},b,c^{-1})$ respectively. Thus
(\ref{e:basis1}) is the basis
$\{m_i^{(-1,1)}\}_{i\in \N}$
of $M_\lambda(a,b,c)$.

We next prove the existence of  $\{m_i^{(1,\sigma)}\}_{i\in \N}$. A direct calculation yields that
$$
(B-\theta_1(\lambda,q;b))
(C^\vee-\theta_0(\lambda,q;c))
m_0
=\phi_1(\lambda,q;c,b,a)m_0.
$$
The central elements $\alpha$, $\beta$, $\gamma$ act on $M_\lambda(a,b,c)$ as scalar multiplications by
$$
\omega(\lambda,q;c,b,a), \qquad \quad
\omega(\lambda,q;a,c,b), \qquad \quad
\omega(\lambda,q;b,a,c)
$$
respectively. By Lemma~\ref{lem:auto} the automorphism $\widetilde \sigma$ of $\triangle$ sends
\begin{gather*}
(A,B,C,\alpha,\beta,\gamma) \quad \mapsto \quad (C^\vee,B,A,\gamma,\beta,\alpha).
\end{gather*}
Let $\{v_i\}_{i\in \N}$ denote the canonical basis of $M_\lambda(c,b,a)$.
By the universal property of $M_\lambda(c,b,a)$ there exists a unique $\triangle$-module homomorphism $M_\lambda(c,b,a)\to M_\lambda(a,b,c)^\sigma$ that sends $v_0$ to $m_0$. Pulling back via $\widetilde\sigma$ we obtain a $\triangle$-module homomorphism
$\iota:M_\lambda(c,b,a)^\sigma\to M_\lambda(a,b,c)$ that sends $v_i$ to
\begin{gather}\label{e:basis2}
\prod_{h=1}^i(C-\theta_{h-1}(\lambda,q;c))m_0
\qquad \quad \hbox{for all $i\in \N$}.
\end{gather}
%By symmetry, there is a $\triangle$-module homomorphism $\iota':M_\lambda(a,b,c)\to M_\lambda(c,b,a)^\sigma$ that sends $m_0$ to $v_0$.
%The universal property of Verma $\triangle$-modules implies that  $\iota\circ\iota'$ and $\iota'\circ\iota$ are the identity maps. %
The tridiagonal matrix $T(\lambda,q;a,b,c)$ representing $C$ with respect to the basis $\{m_i\}_{i\in \N}$ of $M_\lambda(a,b,c)$ has nonzero subdiagonal entries
$$
T(\lambda,q;a,b,c)_{i,i-1}=-\lambda q^{1-2i} b^{-1}
\qquad \quad
\hbox{for all $i\in \N^*$}.
$$
Therefore (\ref{e:basis2}) is a basis of $M_\lambda(a,b,c)$ and $\iota$ is an isomorphism. The matrices representing $C$, $B$ with respect to the basis $\{v_i\}_{i\in \N}$ of $M_\lambda(c,b,a)^\sigma$ are equal to
$L(\lambda^{-1},q^{-1};c^{-1})$, $U(\lambda^{-1},q^{-1};c^{-1},b^{-1},a^{-1})$
respectively. By (\ref{e:Cvee}) the matrix representing $A$ with respect to the basis $\{v_i\}_{i\in \N}$ of $M_\lambda(c,b,a)^\sigma$ is
$T(\lambda^{-1},q^{-1};c^{-1},b^{-1},a^{-1})$. Therefore (\ref{e:basis2}) is the basis $\{m_i^{(1,\sigma)}\}_{i\in \N}$ of $M_\lambda(a,b,c)$.
\end{proof}

%\Needspace{5\baselineskip}

By Lemma~\ref{lem:recrel1} we have

\begin{lem}\label{lem:4TranMat}
For each $(\e,g)\in V_4$ the transition matrices from $\{m_i^{(\e,g)}\}_{i\in \N}$ to $\{m_i^{(-\e,g)}\}_{i\in \N}$ and
$\{m_i^{(\e,g\sigma)}\}_{i\in \N}$ are
$E(\Lambda,Q;X)$ and $S(\Lambda,Q;X,Y,Z)$
with $(\Lambda,Q;X,Y,Z)=(\lambda,q;a,b,c)^{(\e,g)}$, respectively.
\end{lem}

\section{Finite-dimensional irreducible $\triangle$-modules}
\label{section:L(a,b,c)}

We are going to classify the finite-dimensional irreducible $\triangle$-modules when $q$ is not a root of unity. To do this, from now on we set $\lambda=q^n$ and start with a quotient $\triangle$-module of $M_\lambda(a,b,c)$.

\subsection{The quotient $\triangle$-module $V_n(a,b,c)$ of $M_\lambda(a,b,c)$}

Define $N_\lambda(a,b,c)$ to be the $\F$-subspace of $M_\lambda(a,b,c)$ spanned by the $m_{i}$ for all $i\geq n+1$. It is equivalent to say $N_\lambda(a,b,c)=K_a(A)M_\lambda(a,b,c)$, where
$$
K_Y(X)=\prod_{i=0}^n(X-\theta_i(q;Y)).
$$
By construction $N_\lambda(a,b,c)$ is $A$-invariant.
By the setting $\lambda=q^n$ the $(n,n+1)$-entry $\phi_{n+1}(q;a,b,c)$ of $U(\lambda,q;a,b,c)$ is zero. It follows that $N_\lambda(a,b,c)$ is $B$-invariant. Since $\triangle$ is generated by $A$, $B$, $\gamma$ the $\F$-space $N_\lambda(a,b,c)$ is a $\triangle$-module and hence
$$
V_n(a,b,c)=M_\lambda(a,b,c)/N_\lambda(a,b,c)
$$
is a $\triangle$-module of dimension $n+1$. A criterion for an irreducible $\triangle$-module to be isomorphic to $V_n(a,b,c)$ immediately arises from the definition of $V_n(a,b,c)$.

\begin{lem}\label{lem:preiso}
Let $V$ denote an $(n+1)$-dimensional irreducible $\triangle$-module. Assume that
\begin{enumerate}
\item there is a nontrivial $\triangle$-module homomorphism $\iota:M_\lambda(a,b,c)\to V$;

\item $K_a(X)$ is the characteristic polynomial of $A$ on $V$.
\end{enumerate}
Then there is a $\triangle$-module isomorphism $V_n(a,b,c)\to V$ given by
\begin{gather*}
m+N_\lambda(a,b,c)\quad \mapsto \quad \iota(m) \qquad \quad \hbox{for all $m\in M_\lambda(a,b,c)$}.
\end{gather*}
\end{lem}
\begin{proof}
Since $V$ is irreducible $\iota$ is surjective. By the Cayley-Hamilton theorem $K_a(A)$ vanishes on $V$ and thus $N_\lambda(a,b,c)$ is contained in the kernel of $\iota$. This lemma follows.
\end{proof}

For each $(\e,g)\in V_4$ let
$$
v^{(\e,g)}_i=m_i^{(\e,g)}+N_\lambda(a,b,c) \qquad \quad (0\leq i\leq n).
$$
Clearly $\{v_i^{(1,1)}\}^n_{i=0}$ is a basis of $V_n(a,b,c)$. The matrix $E(\lambda,q;a)$ is upper triangular with nonzero diagonal entries and so is $S(\lambda,q;a,b,c)$. Therefore, by Lemma~\ref{lem:4TranMat}, $\{v_i^{(\e,g)}\}^n_{i=0}$ is a basis of $V_n(a,b,c)$ for each $(\e,g)\in V_4$.

\begin{lem}\label{lem:V4bases}
The matrices representing $A$, $B$, $C$ with respect to the basis $\{v_i^{(\e,g)}\}^n_{i=0}$ of $V_n(a,b,c)$ for each $(\e,g)\in V_4$ are as follows:

\begin{table}[H]
\centering
\begin{tabular}{c||c|c|c}
%basis of $V_n(a,b,c)$
&$A$ &$B$ &$C$\\
\hline
$\{v_i^{(1,1)}\}_{i=0}^n$
&$L(q;a)$
&$U(q;a,b,c)$
&$T(q;a,b,c)$\\

$\{v_i^{(-1,1)}\}_{i=0}^n$
&$L(q;a^{-1})$
&$U(q;a^{-1},b,c^{-1})$
&$T(q;a^{-1},b,c^{-1})$\\

$\{v_i^{(1,\sigma)}\}_{i=0}^n$
&$T(q^{-1};c^{-1},b^{-1},a^{-1})$
&$U(q^{-1};c^{-1},b^{-1},a^{-1})$
&$L(q^{-1};c^{-1})$\\

$\{v_i^{(-1,\sigma)}\}_{i=0}^n$
&$T(q^{-1};c,b^{-1},a)$
&$U(q^{-1};c,b^{-1},a)$
&$L(q^{-1};c)$
\end{tabular}
\end{table}
\end{lem}
\begin{proof}
If $\Lambda=Q^n$ then all entries on the $(n+1)$th column of $E(\Lambda,Q;X)$ are zero except for the diagonal entry and so is $S(\Lambda,Q;X,Y,Z)$. Therefore, by Lemma~\ref{lem:4TranMat}, $m_{n+1}^{(\e,g)}\in N_\lambda(a,b,c)$ for each $(\e,g)\in V_4$. The result now follows by Proposition~\ref{lem:M4bases}.
\end{proof}

\begin{lem}\label{lem:chara}
On the $\triangle$-module $V_n(a,b,c)$
\begin{enumerate}
\item the characteristic polynomials of $A$, $B$, $C$ are $K_a(X)$, $K_b(X)$, $K_c(X)$ respectively;

\item the traces of $A$, $B$, $C$ are $[n+1]_q(a+a^{-1})$, $[n+1]_q(b+b^{-1})$, $[n+1]_q(c+c^{-1})$ respectively.
\end{enumerate}
\end{lem}
\begin{proof}
(i) By Lemma~\ref{lem:V4bases} the matrices representing $A$, $B$ with respect to $\{v_i^{(1,1)}\}_{i=0}^n$ are lower and upper bidiagonal with $\{\theta_i(q;a)\}^n_{i=0}$, $\{\theta_i(q;b)\}^n_{i=0}$ on the diagonal entries, respectively. Therefore $K_a(X)$, $K_b(X)$ are the characteristic polynomials of $A$, $B$ on $V_n(a,b,c)$, respectively. By Lemma~\ref{lem:V4bases} the matrix representing $C$ with respect to $\{v_i^{(1,\sigma)}\}^n_{i=0}$ is lower bidiagonal with $\{\theta_i(q;c)\}^n_{i=0}$ on the diagonal entries. Therefore $K_c(X)$ is the characteristic polynomial of $C$ on $V_n(a,b,c)$.

(ii) By (i) we have to show that
$$
\sum^n_{i=0} \theta_i(Q;X)=[n+1]_Q(X+X^{-1}).
$$
To get the equality, apply the summation formula of the geometric series.
%Therefore (ii) follows from (i).
\end{proof}

\subsection{The irreducibility conditions for $V_n(a,b,c)$}

This section is devoted to proving the following necessary and sufficient conditions for $V_n(a,b,c)$ to be irreducible.

\begin{thm}\label{thm:irr}
The $\triangle$-module $V_n(a,b,c)$ is irreducible if and only if the following conditions hold:
\begin{enumerate}
\item $q^{2i}\not=1$ for $1\leq i\leq n$.

\item $abc$, $a^{-1}bc$, $ab^{-1}c$, $abc^{-1}\not\in\{q^{1-n},q^{3-n},\ldots,q^{n-1}\}$.
\end{enumerate}
\end{thm}
\begin{proof}
(Necessity): By Lemma~\ref{lem:V4bases} the matrices representing $A$, $B$ with respect to $\{v_i^{(1,1)}\}^n_{i=0}$ are $L(q;a)$, $U(q;a,b,c)$ respectively. Thus, if there exists $1\leq i\leq n$ such that the $(i-1,i)$-entry $\phi_i(q;a,b,c)$ of $U(q;a,b,c)$ is zero, then $\sum\limits_{h=i}^n \F v_h^{(1,1)}$ is $A$- and $B$-invariant and therefore is a $\triangle$-module, contrary to the irreducibility of $V_n(a,b,c)$.
Therefore
\begin{gather}\label{e:irr1}
\phi_i(q;a,b,c)\not=0 \qquad \quad (1\leq i\leq n).
\end{gather}
The matrices representing $A$, $B$ with respect to $\{v_i^{(-1,1)}\}_{i=0}^n$ are $L(q;a^{-1})$, $U(q;a^{-1},b,c^{-1})$ respectively. Thus a similar argument leads to
\begin{gather}\label{e:irr2}
\phi_i(q;a^{-1},b,c^{-1})\not=0 \qquad \quad (1\leq i\leq n).
\end{gather}
Conditions (i), (ii) are equivalent to (\ref{e:irr1}), (\ref{e:irr2}).

(Sufficiency):
Let
\begin{gather*}
R = \prod_{h=1}^{n}(B-\theta_h(q;b)),\qquad \quad
S_i = \prod_{h=1}^{n-i}(A-\theta_{h-1}(q^{-1};a)) \qquad \quad (0\leq i\leq n).
\end{gather*}
By Lemma~\ref{lem:V4bases} we have $R v \in \F v_0^{(1,1)}$ for all $v\in V_n(a,b,c)$. Let $F$ denote the $(n+1)\times (n+1)$ matrix such that
$$
R\, S_i v_j^{(1,1)}=F_{ij}\, v_0^{(1,1)} \qquad \quad (0\leq i,j\leq n).
$$
A routine calculation shows that $F$ satisfies the conditions given in Lemma~\ref{lem:recrel2}(i) with $(Q;X,Y,Z)=(q;a,b,c)$.
Therefore $F=F(q;a,b,c)$ is lower triangular. By (i), (ii) %the diagonal entries
\begin{equation*}
F_{ii}=\prod_{h=1}^i
\frac{[n-i+h]_{q}}{[n-h+1]_{q}}\, \phi_h(q;a,b,c)
\prod^{n-i}_{h=1}\phi_h(q;a^{-1},b,c)
\qquad \quad (0\leq i\leq n)
\end{equation*}
are nonzero and thus $F$ is invertible.

Now let $V$ denote a nonzero $\triangle$-submodule of $V_n(a,b,c)$. We show that $V=V_n(a,b,c)$. Pick a nonzero vector $v\in V$ and  %let $a_i\in \F$ ($0\leq i\leq n$) such that
consider
\begin{gather}\label{e:irr}
R\, S_i v = a_i v_0^{(1,1)} \qquad \quad (0\leq i\leq n),
\end{gather}
where $a_i\in \F$. Since $V$ is a $\triangle$-module $a_iv_0^{(1,1)}\in V$ for each $0\leq i\leq n$. On the other hand, write $v=\sum\limits_{i=0}^n b_i v_i^{(1,1)}$ where $b_i\in \F$ and express the equations (\ref{e:irr}) as the matrix equation
\begin{equation*}
F(b_0,b_1,\ldots,b_n)^t=(a_0,a_1,\ldots,a_n)^t,
\end{equation*}
where ``$t$'' means the transpose. Since at least one of $b_i$ for all $0\leq i\leq n$ is nonzero, the invertibility of $F$ implies that at least one of $a_i$ for all $0\leq i\leq n$ is nonzero. Therefore $v_0^{(1,1)}\in V$ which generates $V_n(a,b,c)$ as a $\triangle$-module, as claimed.
%Therefore $V$ is equal to $V_n(a,b,c)$.
\end{proof}

\subsection{$24$ bases of irreducible $V_n(a,b,c)$}

When the $\triangle$-module $V_n(a,b,c)$ is irreducible, there are $24$ bases of $V_n(a,b,c)$, described in Proposition~\ref{lem:24}, with respect to which the matrices representing $A$, $B$, $C$ have similar forms to those in Lemma~\ref{lem:V4bases}. Some of these bases will be used in the proofs of Lemma~\ref{lem:iso}, Theorem~\ref{thm:lp} and Theorem~\ref{thm:U}.

Similar to Section~\ref{s:4}, in order to describe the symmetry among the $24$ bases,
we define an action of $\{\pm 1\}^2\rtimes \S_3$ on the set consisting of all $4$-tuples $(q^\e;a,b,c)\in \F^4$ with $a,b,c\not=0$ and $\e\in\{\pm 1\}$ by
\begin{gather*}
\begin{array}{ll}
(q^\e;a,b,c)^{(-1,1,1)} = (q^\e;a^{-1},b,c^{-1}),
\qquad \quad
&(q^\e;a,b,c)^{(1,1,\sigma)} = (q^{-\e};c^{-1},b^{-1},a^{-1}),\\
(q^\e;a,b,c)^{(1,-1,1)} = (q^\e;a,b^{-1},c^{-1}),
\qquad \quad
&(q^\e;a,b,c)^{(1,1,\tau)} =(q^{-\e};b^{-1},a^{-1},c^{-1}),
\end{array}
\end{gather*}
where $\{\pm 1\}^2\rtimes \S_3$ is the semidirect product of $\S_3$ by $\{\pm 1\}^2$ with respect to the group homomorphism $\S_3\to {\rm Aut}\big(\{\pm 1\}^2\big)$ defined by
\begin{gather*}
(-1,1)^\sigma = (-1,1), \qquad  (1,-1)^\sigma = (-1,-1), \qquad
(-1,1)^\tau = (1,-1), \qquad  (1,-1)^\tau = (-1,1).
\end{gather*}

\begin{prop}\label{lem:24}
Assume that the $\triangle$-module $V_n(a,b,c)$ is irreducible. For each $(\e_0,\e_1,g)\in \{\pm 1\}^2\rtimes\S_3$, up to scalar multiplication, there exists a unique basis $\{v_i^{(\e_0,\e_1,g)}\}_{i=0}^n$ of $V_n(a,b,c)$ with respect to which the matrices representing $(A,B,C)^g$ are
$$L(Q;X), \qquad \quad
U(Q;X,Y,Z), \qquad \quad
T(Q;X,Y,Z)
$$
with $(Q;X,Y,Z)=(q;a,b,c)^{(\e_0,\e_1,g)}$, respectively.
\end{prop}
\begin{proof}
(Uniqueness): From the matrix $L(q;a)$ representing $A$ with respect to $\{v_i^{(1,1,1)}\}^n_{i=0}$, we see that
$$
v_i^{(1,1,1)} = \prod_{h=1}^{i}(A-\theta_{h-1}(q;a))v_0^{(1,1,1)} \qquad \quad (1\leq i\leq n).
$$
The first column of $U(q;a,b,c)$ implies that $v_0^{(1,1,1)}$ is an eigenvector of $B$ on $V_n(a,b,c)$ with respect to $\theta_0(q;b)$. The uniqueness of $\{v_i^{(1,1,1)}\}^n_{i=0}$ is now immediate from this lemma:

\begin{lem}\label{lem:eigen1}
Each eigenspace of $A$, $C$ on $V_n(a,b,c)$ is of dimension one. Moreover, if $V_n(a,b,c)$ is irreducible
then each eigenspace of $B$ on $V_n(a,b,c)$ is of dimension one.
\end{lem}
\begin{proof}
By Lemma~\ref{lem:V4bases}
the matrix representing $A$ (resp. $C$) with respect to the basis $\{v_i^{(1,1)}\}^n_{i=0}$ (resp. $\{v_i^{(1,\sigma)}\}_{i=0}^n$) of $V_n(a,b,c)$ is lower bidiagonal with nonzero subdiagonal entries. Therefore each eigenspace of $A$, $C$ on $V_n(a,b,c)$ is of dimension one. Suppose that $V_n(a,b,c)$ is irreducible. The matrix representing $B$ with respect to the basis $\{v_i^{(1,1)}\}^n_{i=0}$ of $V_n(a,b,c)$ is the upper bidiagonal matrix $U(q;a,b,c)$. By Theorem~\ref{thm:irr} the superdiagonal entries of $U(q;a,b,c)$ are nonzero. Therefore each eigenspace of $B$ on $V_n(a,b,c)$ is of dimension one.
\end{proof}
\noindent By similar arguments the uniqueness of $\{v_i^g\}^n_{i=0}$ follows for each $g\in \{\pm 1\}^2\rtimes\S_3$.

(Existence): By Lemma~\ref{lem:V4bases}, $\{v_i^{(\e,g)}\}^n_{i=0}$ is the basis $\{v_i^{(\e,1,g)}\}^n_{i=0}$ of $V_n(a,b,c)$ for each $(\e,g)\in V_4$. Since $\{\pm 1\}^2\rtimes \S_3$ is generated by $(-1,1,1)$, $(1,-1,1)$, $(1,1,\sigma)$, $(1,1,\tau)$ it remains to show the existence of $\{v_i^{(1,-1,1)}\}^n_{i=0}$ and   $\{v_i^{(1,1,\tau)}\}^n_{i=0}$. We first show the existence of $\{v_i^{(1,-1,1)}\}^n_{i=0}$. It is routine to verify that
\begin{gather}\label{e:1-11}
u=\sum_{i=0}^n
\prod\limits_{h=1}^{n-i}\phi_{h}(q;a^{-1},b^{-1},c)
\prod\limits_{h=1}^i (\theta_0(q;b^{-1})-\theta_{h-1}(q^{-1};b^{-1}))
v_i^{(1,1)}
\end{gather}
satisfies $Bu=\theta_0(q;b^{-1})u$ and
$
(B-\theta_1(q;b^{-1}))
(A-\theta_0(q;a))u=\phi_1(q;a,b^{-1},c^{-1})
\hspace{0.5mm}u.
$
The central elements $\alpha$, $\beta$, $\gamma$ act on $V_n(a,b,c)$ as scalar multiplications by
$$
\omega(q;b^{-1},c^{-1},a), \qquad \quad
\omega(q;c^{-1},a,b^{-1}), \qquad \quad
\omega(q;a,b^{-1},c^{-1})
$$
respectively. Let $\{u_i\}_{i\in \N}$ denote the canonical basis of $M_\lambda(a,b^{-1},c^{-1})$. Combining the above comments, the universal property of $M_\lambda(a,b^{-1},c^{-1})$ implies that there is a $\triangle$-module homomorphism
$M_\lambda(a,b^{-1},c^{-1})\to V_n(a,b,c)$ that sends $u_0\mapsto u$.
Recall from Lemma~\ref{lem:chara}(i) that $K_a(X)$ is the characteristic polynomial of $A$ on $V_n(a,b,c)$. By Lemma~\ref{lem:preiso} there is a $\triangle$-module isomorphism $V_n(a,b^{-1},c^{-1})\to V_n(a,b,c)$ that sends $u_i+N_\lambda(a,b^{-1},c^{-1})$ to
\begin{gather}\label{e:basis3}
\prod_{h=1}^i(A-\theta_{h-1}(q;a))u
\qquad \quad (0\leq i\leq n).
\end{gather}
By Lemma~\ref{lem:V4bases}, with respect to the basis $\{u_i+N_\lambda(a,b^{-1},c^{-1})\}^n_{i=0}$ of $V_n(a,b^{-1},c^{-1})$ the matrices representing $A$, $B$, $C$  are $L(q;a)$, $U(q;a,b^{-1},c^{-1})$, $T(q;a,b^{-1},c^{-1})$ respectively.
Therefore (\ref{e:basis3}) is the basis $\{v_i^{(1,-1,1)}\}^n_{i=0}$ of $V_n(a,b,c)$.

We next prove the existence of  $\{v_i^{(1,1,\tau)}\}^n_{i=0}$. A direct calculation yields that
\begin{gather}\label{e:11tau}
v=\sum_{i=0}^n
\prod\limits_{h=1}^{i}(\theta_0(q;a)-\theta_{h-1}(q^{-1};a)) v_{n-i}^{(1,1)}
\end{gather}
satisfies $Av=\theta_0(q;a) v$ and
$
(A-\theta_1(q;a))(B-\theta_0(q;b))v=\phi_1(q;b,a,c)v.
$
The central elements $\alpha$, $\beta$, $\gamma$ act on $V_n(a,b,c)$ as scalar multiplications by
$$
\omega(q;c,b,a), \qquad \quad
\omega(q;a,c,b), \qquad \quad
\omega(q;b,a,c)
$$
respectively.
Recall from Lemma~\ref{lem:auto} that the automorphism $\widetilde\tau$ of $\triangle$ maps
\begin{gather*}
(A,B,C,\alpha,\beta,\gamma) \quad \mapsto \quad (B,A,C^\vee,\beta,\alpha,\gamma).
\end{gather*}
Let $\{v_i\}_{i\in \N}$ denote the canonical basis of $M_\lambda(b,a,c)$. Combining the above comments, the universal property of $M_\lambda(b,a,c)$ implies that there exists a $\triangle$-module homomorphism $M_\lambda(b,a,c) \to V_n(a,b,c)^{\tau}$ that maps $v_0$ to $v$. By Lemma~\ref{lem:chara}(i), $K_b(X)$ is the characteristic polynomial of $A$ on $V_n(a,b,c)^{\tau}$. By Lemma~\ref{lem:preiso} there is a $\triangle$-module isomorphism $V_n(b,a,c)\to V_n(a,b,c)^\tau$ that maps $v_0+N_\lambda(b,a,c)$ to $v$. Pulling back via $\widetilde \tau$ we obtain a $\triangle$-module isomorphism $V_n(b,a,c)^\tau \to V_n(a,b,c)$ that sends $v_i+N_\lambda(b,a,c)$ to
\begin{gather}\label{e:basis4}
\prod_{h=1}^i (B-\theta_{h-1}(q;b))v
\qquad \quad (0\leq i\leq n).
\end{gather}
By Lemma~\ref{lem:V4bases} the matrices representing $B$, $A$ with respect to the basis $\{v_i+N_\lambda(b,a,c)\}^n_{i=0}$ of $V_n(b,a,c)^\tau$ are
$L(q^{-1};b^{-1})$, $U(q^{-1};b^{-1},a^{-1},c^{-1})$
respectively. By (\ref{e:Cvee}) the matrix representing $C$ with respect to $\{v_i+N_\lambda(b,a,c)\}^n_{i=0}$ is
$T(q^{-1};b^{-1},a^{-1},c^{-1})$.
Therefore (\ref{e:basis4}) is the basis $\{v_i^{(1,1,\tau)}\}_{i=0}^n$ of $V_n(a,b,c)$.
\end{proof}

We end this section with the remark that for each $g\in \{\pm1\}^2\rtimes \S_3$, up to scalar multiplication the transition matrices from $\{v_i^g\}_{i=0}^n$ to
$\{v_i^{g(-1,1,1)}\}_{i=0}^n$,
$\{v_i^{g(1,-1,1)}\}_{i=0}^n$,
$\{v_i^{g(1,1,\sigma)}\}_{i=0}^n$,
$\{v_i^{g(1,1,\tau)}\}_{i=0}^n$
are
$$
E(Q;X), \qquad
F(Q;X,Y^{-1},Z), \qquad
S(Q;X,Y,Z), \qquad
P(Q;X,Y,Z)E(Q^{-1};Y)
$$
with $(Q;X,Y,Z)=(q;a,b,c)^g$, respectively. To see this, without loss we assume that $g=(1,1,1)$. By Lemma~\ref{lem:4TranMat} the transition matrices from $\{v_i^{(1,1,1)}\}_{i=0}^n$ to $\{v_i^{(-1,1,1)}\}_{i=0}^n$ and $\{v_i^{(1,1,\sigma)}\}_{i=0}^n$ are as claimed. Let $F$ denote the transition matrix from $\{v_i^{(1,1,1)}\}_{i=0}^n$ to $\{v_i^{(1,-1,1)}\}_{i=0}^n$. Since $F_{i0}$ for each $0\leq i\leq n$ can be chosen as the coefficient of $v_i^{(1,1)}$ in (\ref{e:1-11}), the matrix $F$ satisfies the conditions given in Lemma~\ref{lem:recrel2}(i) with $(Q;X,Y,Z)=(q;a,b^{-1},c)$. Therefore $F=F(q;a,b^{-1},c)$. To get the transition matrix from
$\{v_i^{(1,1,1)}\}_{i=0}^n$ to $\{v_i^{(1,1,\tau)}\}_{i=0}^n$, we decompose it as the product of the transition matrix $P$ from $\{v_i^{(1,1,1)}\}_{i=0}^n$ to $\{v_i^{(1,-1,\tau)}\}_{i=0}^n$ with the transition matrix $E(q^{-1};b)$ from
$\{v_i^{(1,-1,\tau)}\}_{i=0}^n$ to $\{v_i^{(1,1,\tau)}\}_{i=0}^n$. Since $P_{n-i,0}$ for each $0\leq i\leq n$ can be chosen as the coefficient of $v_{n-i}^{(1,1)}$ in (\ref{e:11tau}), the matrix $P$ satisfies the conditions given in Lemma~\ref{lem:recrel2}(ii) with $(Q;X,Y,Z)=(q;a,b,c)$. Therefore $P=P(q;a,b,c)$.

\subsection{Finite-dimensional irreducible $\triangle$-modules at $q$ not a root of unity}\label{s:class}

We are now in the position to prove the main result of this paper:

\begin{thm}\label{thm:class}
Assume that $q$ is not a root of unity. Let $\M$ denote the set of the isomorphism classes of irreducible $\triangle$-modules that have dimension $n+1$. Let $\T$ denote the set of all triples $(a,b,c)$ of nonzero scalars in $\F$ that satisfy Theorem~{\rm\ref{thm:irr}(ii)}:
$$
abc,\;
a^{-1}bc,\;
ab^{-1}c,\;
abc^{-1}
\notin\{q^{1-n},q^{3-n},\ldots,q^{n-1}\}.
$$
Define an action of the group $\{\pm1\}^3$ on $\T$ by
\begin{equation*}
(a,b,c)^{(-1,1,1)}=(a^{-1},b,c), \qquad \quad
(a,b,c)^{(1,-1,1)}=(a,b^{-1},c), \qquad \quad
(a,b,c)^{(1,1,-1)}=(a,b,c^{-1})
\end{equation*}
for all $(a,b,c)\in \T$. Let $\T/\{\pm 1\}^3$ denote the set of the $\{\pm1\}^3$-orbits of $\T$.
For $(a,b,c)\in \T$ let $[a,b,c]$ denote the $\{\pm 1\}^3$-orbit of $\T$ that contains $(a,b,c)$.
Then there is a bijection $\T/\{\pm1\}^3\to \M$ given by
\begin{equation*}
[a,b,c] \quad \mapsto \quad \hbox{the isomorphism class of $V_n(a,b,c)$}
\qquad \quad \hbox{for $[a,b,c]\in \T/\{\pm1\}^3$}.
\end{equation*}
\end{thm}
\begin{proof}
We begin with a lemma, the ``if'' part of which implies the existence of the map $\T/\{\pm1\}^3\to \M$ and the ``only if'' part implies that the map is injective.

\begin{lem}\label{lem:iso}
Assume that $q^{2i}\not=1$ for $1\leq i\leq n$.
For $(a,b,c)$ and $(r,s,t)$ in $\T$,
$V_n(a,b,c)$ is isomorphic to $V_n(r,s,t)$ if and only if $[a,b,c]=[r,s,t]$.
\end{lem}
\begin{proof}
(Necessity): Any two isomorphic finite-dimensional $\triangle$-modules have the same trace map.
%Two isomorphic finite-dimensional $\triangle$-modules $V_n(a,b,c)$ and $V_n(r,s,t)$ have the same trace map.
By Lemma~\ref{lem:chara}(ii) this part follows.

(Sufficiency):
It suffices to show that $V_n(a,b,c)$ is isomorphic to $V_n(a^{-1},b,c)$, $V_n(a,b^{-1},c)$, $V_n(a,b,c^{-1})$. Observe that $U(Q;X,Y,Z)$ and $T(Q;X,Y,Z)$ are invariant when $Z$ is replaced by $Z^{-1}$.
Therefore, by Lemma~\ref{lem:V4bases} with $(\e,g)=(1,1)$ and $(\e,g)=(-1,1)$ we see that $V_n(a,b,c)$ is isomorphic to $V_n(a,b,c^{-1})$ and $V_n(a^{-1},b,c)$, respectively. By Proposition~\ref{lem:24} with $(\e_0,\e_1,g)=(1,-1,1)$ the $\triangle$-module $V_n(a,b,c)$ is isomorphic to $V_n(a,b^{-1},c)$.
\end{proof}

To see that the map is surjective, we assume that $V$ is an $(n+1)$-dimensional irreducible $\triangle$-module and show that there exists $(a,b,c)\in \T$ such that $V$ is isomorphic to $V_n(a,b,c)$. For any $\theta\in \F$ and any $S\in \triangle$ define $V_S(\theta)=\{v\in V~|~ S v=\theta v\}$. The nonzero scalar $b\in \F$ is chosen to satisfy \begin{equation}\label{bound_condition}
V_B(\theta_{-1}(q;b))=0,
\qquad \quad
V_B(\theta_0(q;b))\not=0.
\end{equation}
To see the existence of $b$ one may apply this lemma:

\begin{lem}\label{lem:theta}
For $i$, $j$ in $\Z$,
$\theta_i(Q;X)=\theta_j(Q;X)$ if and only if $Q^{2i}=Q^{2j}$ or $X^2= Q^{2(n-i-j)}$. %for any $i,j\in \Z$.
\end{lem}
\begin{proof}
Factor $\theta_i(Q;X)-\theta_j(Q;X)=(Q^{i-j}-Q^{j-i})
(Q^{i+j-n}X-Q^{n-i-j}X^{-1})$.
%Since the first factor is nonzero the result follows.
\end{proof}

\noindent Pick an eigenvalue $\theta$ of $B$ on $V$. Since $\F$ is algebraically closed there exists a nonzero scalar $s\in \F$ such that $\theta=\theta_0(q;s)$. Under the assumption that $q$ is not a root of unity, Lemma~\ref{lem:theta} implies that either $\{\theta_i(q;s)\}_{i\in \Z}$ are pairwise distinct or $\theta_i(q;s)=\theta_j(q;s)$ if and only if $i+j$ is equal to a specific integer. In either case there are only finitely many $i\in \Z$ with $V_B(\theta_i(q;s))\not=0$. Therefore there exists $i\in \Z$ with $V_B(\theta_{i-1}(q;s))=0$ and $V_B(\theta_i(q;s))\not=0$; in other words the scalar $b=s q^{2i}$ satisfies (\ref{bound_condition}).
%\begin{equation}\label{bound_condition}
%V_B(\theta_{-1}(q;b))=0,
%\qquad \quad
%V_B(\theta_0(q;b))\not=0.
%\end{equation}
Similarly the nonzero scalar $a\in \F$ is chosen to satisfy
\begin{equation*}
V_A(\theta_{-1}(q;a))=0,
\qquad \quad
V_A(\theta_0(q;a))\not=0.
\end{equation*}
Note that every central element in $\triangle$ acts on $V$ as a scalar.
Since $\F$ is algebraically closed and $q^{2n+2}\not=-1$, there exists a nonzero scalar $c\in \F$ such that $\gamma$ acts on $V$ as $\omega(q;a,b,c)$.

We invoke Lemma~\ref{lem:preiso} to show that $V_n(a,b,c)$ is isomorphic to $V$. Taking the commutator with $B$ on either side of (\ref{a->ABc}), we obtain that
\begin{gather}\label{e:serre}
B^3A-[3]_q B^2AB+[3]_q BAB^2-AB^3=-(q^2-q^{-2})^2(BA-AB). %\label{presentationABc:eq1}
\end{gather}
Pick any vector $v\in V_B(\theta_0(q;b))$. Applying $v$ to either side of (\ref{e:serre}),
%(\ref{presentationABc:eq1}),
we find that $(B-\theta_{-1}(q;b))(B-\theta_0(q;b))(B-\theta_1(q;b))$ vanishes at $Av$. By (\ref{bound_condition}) it reduces to $(B-\theta_0(q;b))(B-\theta_1(q;b)) A v=0$.
This shows that $V_B(\theta_0(q;b))$ is $(B-\theta_1(q;b))A$-invariant. Let $v_0$ denote an eigenvector of $(B-\theta_1(q;b))A$ in $V_B(\theta_0(q;b))$. Recurrently define
\begin{gather}
v_i =
(A-\theta_{i-1}(q;a))v_{i-1} \qquad \quad \hbox{for all $i\in \N^*$}.
\label{irr_basis:2nd}
%w_i&=&(B-\theta_{i-1}(q;b))w_{i-1} \qquad \quad \hbox{for all $i\in \N^*$}.\notag %\label{irr_basis:1st}
\end{gather}
We now proceed by induction to show that
\begin{equation}
(B-\theta_i(q;b)) v_i \in \sum_{h=0}^{i-1}\F v_h \qquad \quad \hbox{for $i\in \N$}. \label{A:uppertriadiagonal}
\end{equation}
By the choice of $v_0$, (\ref{A:uppertriadiagonal}) holds for $i=0,1$. Suppose $i\geq 2$. Observe that
\begin{gather}\label{e:theta-3}
\theta_i(Q;X)=(Q^2+Q^{-2})\theta_{i-1}(Q;X)-\theta_{i-2}(Q;X).
\end{gather}
To get (\ref{A:uppertriadiagonal}), we
apply $v_{i-2}$ to either side of (\ref{b->ABc}) and simplify the resulting equation by using (\ref{irr_basis:2nd}), (\ref{e:theta-3}) and induction hypothesis. To see that $\{v_i\}^n_{i=0}$ is a basis of $V$, we suppose on the contrary that there exists $0\leq i\leq n-1$ such that $v_{i+1}\in \sum\limits_{h=0}^i \F v_h$. Then $\sum\limits_{h=0}^i \F v_h$ is $A$- and $B$-invariant by (\ref{irr_basis:2nd}) and (\ref{A:uppertriadiagonal}), respectively. This leads to a contradiction to the irreducibility of $V$.

Similarly, choose $w_0$ to be an eigenvector of $(A-\theta_1(q;a))B$ in $V_A(\theta_0(q;a))$ and recurrently define
\begin{gather*}
w_i=
(B-\theta_{i-1}(q;b))w_{i-1} \qquad \quad
\hbox{for all $i\in \N^*$}.
\end{gather*}
Then $\{w_i\}^n_{i=0}$ is a basis of $V$ with respect to which the matrix representing $A$ is upper triangular with $\theta_i(q;a)$ on the $(i,i)$-entry for each $0\leq i\leq n$. Therefore
$K_a(X)$ is the characteristic polynomial of $A$ on $V$. By the Cayley-Hamilton theorem $K_a(A)$ vanishes on $V$. In particular $K_a(A)v_0=0$ and hence
\begin{equation}\label{e:vn}
A v_n=\theta_n(q;a) v_n.
\end{equation}

For each $1\leq i\leq n$ let $\phi_i$ denote the $(i-1,i)$-entry of the matrix representing $B$ with respect to the basis $\{v_i\}^n_{i=0}$ of $V$.
%Let $\{\phi_i\}^n_{i=1}$ denote the sequence of scalars in $\F$ such that
%\begin{equation*}
%(B-\theta_i(q;b))v_i \in  \phi_i v_{i-1}+\sum_{h=0}^{i-2}\F v_h \qquad \quad (1\leq i\leq n).
%\end{equation*}
%Let $1\leq i\leq n$ be given.
Apply $v_{i-1}$ $(1\leq i\leq n)$ to either side of (\ref{b->ABc}) and simplify the resulting equation by using (\ref{irr_basis:2nd}), (\ref{e:vn}). Comparing the coefficient of $v_i$ on either side we then obtain that $\phi_{i+1}-(q^2+q^{-2})\phi_i+\phi_{i-1}$ is equal to
\begin{align*}
&(q^2+q^{-2})
(\theta_i(q;a)\theta_{i}(q;b)+\theta_{i-1}(q;a)\theta_{i-1}(q;b))
\\
&\qquad
-\;(\theta_i(q;a)+\theta_{i-1}(q;a))
(\theta_i(q;b)+\theta_{i-1}(q;b))
-(q-q^{-1})^2\omega(q;a,b,c)
\end{align*}
for each $1\leq i\leq n$, where $\phi_0$ and $\phi_{n+1}$ are interpreted as $0$.
A direct calculation yields that $\phi_i=\phi_i(q;a,b,c)$ for all $1\leq i\leq n$ satisfy the above recurrence relation. Since %$q^4\not=1$ and
$q^{4n+4}\not=1$ there are no other scalars $\phi_i$ $(1\leq i\leq n)$ satisfying the recurrence relation. In particular we have
\begin{equation}\label{K3}
(B-\theta_1(q;b))(A-\theta_0(q;a))v_0=\phi_1(q;a,b,c)v_0.
\end{equation}
%To prove that $\alpha v_0=\omega(q;b,c,a) v_0$, we may apply $v_0$ to either side of (\ref{a->ABc}) and use (\ref{K3}) to simplify the resulting equation.
Applying $v_0$ to either side of (\ref{a->ABc}) and using (\ref{K3}) to simplify the resulting equation, we see that $\alpha v_0=\omega(q;b,c,a) v_0$.
A similar argument shows that $\beta w_0=\omega(q;c,a,b) w_0$. Therefore the central elements $\alpha$, $\beta$ act on $V$ as the scalars $\omega(q;b,c,a)$, $\omega(q;c,a,b)$ respectively.

Combining the above comments, the universal property of $M_\lambda(a,b,c)$ implies that there exists a unique $\triangle$-module homomorphism
$M_\lambda(a,b,c)\to V$ that maps $m_0$ to $v_0$. By Lemma~\ref{lem:preiso} this induces a $\triangle$-module isomorphism $V_n(a,b,c)\to V$. By Theorem~\ref{thm:irr} the triple $(a,b,c)\in \T$, as claimed.
\end{proof}

%\Needspace{8\baselineskip}

%Given a finite-dimensional irreducible $\triangle$-module at $q$ not a root of unity, Lemma~\ref{lem:chara}(ii) provides a way to seek the corresponding $a$, $b$, $c$.
As a consequence of Lemma~\ref{lem:chara}(ii) and Theorem~\ref{thm:class} we see that

\begin{cor}\label{cor:abc}
Assume that $q$ is not a root of unity.
Let $V$ denote an $(n+1)$-dimensional irreducible $\triangle$-module. Let ${\rm tr}\hspace{0.05mm}A$, ${\rm tr}\hspace{0.05mm}B$, ${\rm tr}\hspace{0.05mm}C$ denote the traces of $A$, $B$, $C$ on $V$ respectively.
Then $V$ is isomorphic to $V_n(a,b,c)$ if and only if $a$, $b$, $c$ are the roots of the quadratic polynomials
\begin{gather*}
[n+1]_qX^2- {\rm tr}\hspace{0.05mm}A\hspace{0.15mm}X +[n+1]_q,\\
[n+1]_qX^2- {\rm tr}\hspace{0.05mm}B\hspace{0.15mm}X +[n+1]_q,\\
[n+1]_qX^2- {\rm tr}\hspace{0.05mm}C\hspace{0.15mm}X +[n+1]_q,
\end{gather*}
respectively.
%\begin{enumerate}
%\item $V$ is isomorphic to $V_n(a,b,c)$.
%\item $a$, $b$, $c$ are the roots of the quadratic polynomials
%\begin{gather*}
%[n+1]_qX^2- {\rm tr}\,A\; X +[n+1]_q,\\
%[n+1]_qX^2- {\rm tr}\,B\; X +[n+1]_q,\\
%[n+1]_qX^2- {\rm tr}\,C\; X +[n+1]_q,
%\end{gather*}
%respectively.
%\end{enumerate}
\end{cor}

\section{Applications}

On the finite-dimensional $\AW$-module constructed by Zhedanov in \cite[Section~2]{hidden_sym}, the elements $K_0$, $K_1$ act like a Leonard pair but in a less mathematically rigorous treatment. Motivated by the $P$- and $Q$-polynomial association schemes \cite[Section~3.5]{AC1:1984}, Terwilliger independently introduced Leonard pairs \cite[Definition~1.1]{lp2001}. The classification of Leonard pairs \cite[Theorem~1.9]{lp2001} can be considered to be a linear algebraic version of Leonard theorem \cite{Leo1982}, which gave a characterization of the $q$-Racah and related polynomials in the Askey scheme \cite[Chapter~3]{Koe2010}.
This led to the name of Leonard pairs. Afterward the notion of Leonard pairs was naturally extended to Leonard triples by Curtin in \cite[Definition~1.2]{cur2007}.

For the rest assume that $q$ is not a root of unity. To study the finite-dimensional irreducible $\triangle$-modules, it is now enough to consider the $\triangle$-module $V_n(a,b,c)$ for $(a,b,c)\in \T$ by Theorem~\ref{thm:class}.
In Section~\ref{s:lp} we give the equivalent conditions for $A$, $B$, $C$ on $V_n(a,b,c)$ as Leonard pairs or a Leonard triple. In Section~\ref{s:unitary} we discuss the sufficient conditions for $V_n(a,b,c)$ to be unitary under the constraint of $A$, $B$, $C$ as a Leonard triple. In Section~\ref{s:feedback} we apply Theorem~\ref{thm:class} to classify the finite-dimensional irreducible $\V$-modules and compare the result with \cite[Theorem~4]{hp2001}. In \cite{lp&equ2011} Alnajjar described how to obtain the Leonard pairs of $q$-Racah and other $q$-types from $\U$-modules. Improving the result Terwilliger gave an $\F$-algebra homomorphism $\triangle\to \U$ below \cite[Proposition~1.1]{uaw&equit2011}. %Gao and Hou \cite{Hou2012} made use of this homomorphism to construct the Leonard triples of $q$-Racah type.
The purpose of Section~\ref{s:U} is to determine how many $\U$-modules on $V_n(a,b,c)$ give the $\triangle$-module $V_n(a,b,c)$ by pulling back via the homomorphism.
The work \cite{gz92} of Granovski{\u\i} and Zhedanov showed how the Racah coefficients of $\mathfrak{su}_q(2)$ are relevant to the $\AW$-modules. The idea was recently extended to $\mathfrak{sl}_{-1}(2)$ and $\mathfrak{osp}_q(1|2)$ in \cite{gvz2013,gvz2015} respectively. Inspired by the performances, the end of Section~\ref{s:Racah} is to display an $\F$-algebra homomorphism $\triangle\to \U\otimes\U\otimes\U$ and explain its connection to the Racah coefficients of $\U$.

\subsection{Leonard pairs and Leonard triples}\label{s:lp}

A tridiagonal matrix is said to be {\it irreducible} if each entry on the subdiagonal and superdiagonal is nonzero. Let $V$ denote a nonzero finite-dimensional vector space.
A pair (resp.\! triple) of linear transformations on $V$ is called a {\it Leonard pair} (resp.\! {\it triple}) whenever for each of the two (resp.\! three) transformations, there exists a basis of $V$ with respect to which the matrix representing that transformation is diagonal and the matrices representing the other transformations are irreducible tridiagonal.

\begin{lem}\label{lem:diag}
For $(a,b,c)\in \T$ the following are equivalent:
\begin{enumerate}
\item $A$ {\rm(}resp. $B${\rm) } {\rm(}resp. $C${\rm) } is diagonalizable on $V_n(a,b,c)$.

\item $\theta_i(q;a)$
    {\rm(}resp. $\theta_i(q;b)${\rm) }
    {\rm(}resp. $\theta_i(q;c)${\rm) } for all $0\leq i\leq n$  are pairwise distinct.

\item $a^2$
{\rm(}resp. $b^2${\rm)}
{\rm(}resp. $c^2${\rm)}
is not among $q^{2n-2},q^{2n-4},\ldots,q^{2-2n}$.
\end{enumerate}
\end{lem}
\begin{proof}
(i) $\Leftrightarrow$ (ii) follows from Lemma~\ref{lem:eigen1}.
(ii) $\Leftrightarrow$ (iii) follows from Lemma~\ref{lem:theta}.
\end{proof}

\begin{thm}\label{thm:lp}
For $(a,b,c)\in \T$ the following are equivalent:
\begin{enumerate}
\item $A$, $B$
{\rm(}resp. $A$, $C${\rm) }
{\rm(}resp. $B$, $C${\rm) }
act on $V_n(a,b,c)$ as a Leonard pair.

\item $A$, $B$
{\rm(}resp. $A$, $C${\rm) }
{\rm(}resp. $B$, $C${\rm) } are diagonalizable on $V_n(a,b,c)$.

\item Neither of $a^2$, $b^2$
{\rm(}resp. $a^2$, $c^2${\rm) }
{\rm(}resp. $b^2$, $c^2${\rm) }
is among $q^{2n-2},q^{2n-4},\ldots,q^{2-2n}$.
\end{enumerate}
\end{thm}
\begin{proof}
(ii) $\Leftrightarrow$ (iii) follows from Lemma~\ref{lem:diag}.
By Proposition~\ref{lem:24} the matrices representing $B$, $C$ with respect to the basis $\{v_i^{(1,1,\sigma\tau)}\}^n_{i=0}$ of $V_n(a,b,c)$ are $L(q;b)$, $U(q;b,c,a)$ respectively. By \cite[Lemma~7.3]{LTqracah2012} and \cite[Theorem~7.2]{ter2006}, (i) $\Leftrightarrow$ (iii) holds for the pair $B$, $C$. By similar arguments (i) $\Leftrightarrow$ (iii) holds for the pairs $A$, $B$ and $A$, $C$.
\end{proof}

%\Needspace{5\baselineskip}

\noindent As a consequence of \cite[Theorem~14.5]{LTqracah2012} and Theorem~\ref{thm:lp} we have

\begin{thm}\label{thm:lt}
For $(a,b,c)\in \T$ the following are equivalent:
\begin{enumerate}
\item $A$, $B$, $C$ act on $V_n(a,b,c)$ as a Leonard triple.

\item Any two of $A$, $B$, $C$ act on $V_n(a,b,c)$ as a Leonard pair.

\item $A$, $B$, $C$ are diagonalizable on $V_n(a,b,c)$.

\item None of $a^2$, $b^2$, $c^2$ is among $q^{2n-2}, q^{2n-4},\ldots,q^{2-2n}$.
\end{enumerate}
\end{thm}

\subsection{The unitary structure on $V_n(a,b,c)$}\label{s:unitary}

In this section we assume that there is an involution $*:\F\to \F$. Recall that given an $\F$-vector space $V$ a map $(\,,):V\times V\to \F$ is called a {\it $*$-form} if
\begin{align*}
\begin{array}{ll}
(u+v,w) = (u,w)+(v,w),
\qquad \quad
&(u,v+w) = (u,v)+(u,w),
\\
(\mu u,v) =\mu(u,v),
\qquad \quad
&(u,\mu v) =\mu^*(u,v)
\end{array}
\end{align*}
for all $u,v,w\in V$ and $\mu\in \F$. A $*$-form $(\,,):V\times V\to \F$ is said to be {\it degenerate} if there exists a nonzero vector $u\in V$ such that $(u,v)=0$ for all $v\in V$ or $(v,u)=0$ for all $v\in V$. For example
the Hermitian form is a nondegenerate $*$-form over the complex number field with $*$ as complex conjugation.
%the complex conjugation is an involution of the complex number field and a hermitian form is a $*$
Given an $\F$-algebra $\mathcal A$, a {\it $*$-involution} $\dag$ on $\mathcal A$ means a map $\dag:\mathcal A\to \mathcal A$ satisfying
\begin{align*}
\begin{array}{lll}
(R+S)^\dag=R^\dag+S^\dag,
\qquad \quad
&(RS)^\dag=S^\dag R^\dag,
\qquad \quad
1^\dag=1,
\\
(\mu R)^\dag=\mu^* R^\dag,
\qquad \quad
&(R^\dag)^\dag=R
\end{array}
\end{align*}
for all $R$, $S\in \mathcal A$ and all $\mu\in \F$. Given an $\F$-algebra $\mathcal A$ endowed with a $*$-involution $\dag$, an $\mathcal A$-module $V$ is said to be {\it $\dag$-unitary} or {\it unitary} \cite[Definition~2.3.1]{rosen1992}
if there exists a nondegenerate $*$-form $(\,,)$ with
\begin{gather*}
(S u,v)=(u, S^\dag v)
\qquad \quad
\hbox{for all $u,v\in V$ and $S\in \mathcal A$}.
\end{gather*}

Under some additional hypotheses on $q$, $a$, $b$, $c$ we have a unitary structure on the irreducible $\triangle$-module $V_n(a,b,c)$ with two of $A$, $B$, $C$ as a Leonard pair. Without loss of generality say the action of $A$, $B$ on $V_n(a,b,c)$ as a Leonard pair. Under the assumption $q^*=q^{-1}$, Lemma~\ref{lem:auto} with $g=\tau$ implies that there is a unique $*$-involution $\dag:\triangle\to \triangle$ that maps
$$
(A,B,C,\alpha,\beta,\gamma)
\quad \mapsto \quad
(B, A, C^\vee,\beta,\alpha,\gamma).
$$
For each $0\leq i\leq n$, let $u_i$ and $v_i$ denote the eigenvectors of $A$, $B$ on $V_n(a,b,c)$ with respect to $\theta_i(q;a)$ and $\theta_i(q;b)$ respectively. By Theorem~\ref{thm:lp}, $\{u_i\}_{i=0}^n$ and $\{v_i\}_{i=0}^n$ are two bases of $V_n(a,b,c)$.
If $a^*=b^{-1}$ and $c^*\in \{c,c^{-1}\}$ then it is routine to check that $V_n(a,b,c)$ is $\dag$-unitary with respect to the nondegenerate $*$-form $(\,,)$ defined by
$$
(u_i,v_j)=\left\{
\begin{array}{ll}
1 \qquad \quad &\hbox{if $i=j$},\\
0 \qquad \quad &\hbox{else}.
\end{array}
\right.
$$

By virtue of Theorem~\ref{thm:class} we can conclude that the irreducible $\triangle$-module $V_n(a,b,c)$ with $A$, $B$, $C$ as a Leonard triple admits a unitary structure, provided that $q^*=q^{-1}$ and one of the following holds:
\begin{enumerate}
\item[$\bullet$] $a^*\in\{b,b^{-1}\}$ and $c^*\in \{c,c^{-1}\}$.

\item[$\bullet$] $b^*\in \{c,c^{-1}\}$ and $a^*\in\{a,a^{-1}\}$.

\item[$\bullet$] $c^*\in \{a,a^{-1}\}$ and $b^*\in\{b,b^{-1}\}$.
\end{enumerate}

\subsection{A feedback to $\V$-modules}\label{s:feedback}

Recall from Introduction that $\V$ is a long-studied example of $\AW$ with the defining relations 
\begin{eqnarray*}
q K_1 K_2-q^{-1} K_2 K_1&=& K_0,\\
q K_2 K_0-q^{-1} K_0 K_2&=& K_1,\\
q K_0 K_1-q^{-1} K_1 K_0&=& K_2.
\end{eqnarray*}
The quantum group $\V$ is not the Drinfeld-Jimbo type but plays the crucial roles in the nuclear
spectroscopy \cite{hkp1999},
the quantum Laplace operator \cite{ik2001,nuw:1996},
the $(2 + 1)$-dimensional quantum gravity
\cite{nrz:1990,nr:1993} and so on. In this section we restrict Theorem~\ref{thm:class} to the case of $\V$ and compare the consequence with the extant classification \cite[Theorem~4]{hp2001} of the finite-dimensional irreducible $\V$-modules.

By the universal property of $\triangle$ there is a unique $\F$-algebra homomorphism $\triangle \to \V$ that sends
\begin{eqnarray*}
(A,B,C,\alpha,\beta,\gamma)
&\mapsto&
\left(
(q^{-2}-q^2)K_0,
(q^{-2}-q^2)K_1,
(q^{-2}-q^2)K_2,
0,0,0
\right).
\end{eqnarray*}
Assume that $V$ is an $(n+1)$-dimensional irreducible $\V$-modules. The pullback of $V$ via the above homomorphism gives an irreducible $\triangle$-module structure on $V$. By Theorem~\ref{thm:class} there exists a unique $[a,b,c]\in \T/\{\pm 1\}^3$ such that $V$ is isomorphic to $V_n(a,b,c)$ as a $\triangle$-module. Solving for $a$, $b$, $c$ with the vanishment of $\alpha$, $\beta$, $\gamma$ on $V_n(a,b,c)$ it yields that two possible cases:

\begin{enumerate}
\item[$\bullet$] $a^2=b^2=c^2=-1$.

\item[$\bullet$] $a+a^{-1}=-\e_0(q^{n+1}+q^{-n-1})$, $b+b^{-1}=-\e_1(q^{n+1}+q^{-n-1})$, $c+c^{-1}=-\e_2(q^{n+1}+q^{-n-1})$ where $\e_0,\e_1,\e_2\in\{\pm 1\}$ with $\e_i=\e_{i-1}\e_{i+1}$ for all $i\in \Z/3\Z$.
\end{enumerate}
In terminology of \cite[Theorem~4]{hp2001}, Corollary~\ref{cor:abc} implies that the two cases correspond to the classical and nonclassical irreducible $\V$-modules %and the nonclassical irreducible $\V$-module %of type $(1,-1)$ and types $(1,1)$, $(-1,1)$, $(-1,-1)$
respectively. By Theorem~\ref{thm:lt} the action of $K_0$, $K_1$, $K_2$ on each finite-dimensional irreducible $\V$-module forms a Leonard triple. By the result of Section~\ref{s:unitary} the finite-dimensional irreducible $\V$-modules are unitary when $q^*=q^{-1}$.

\subsection{A connection to $\U$-modules}\label{s:U}

The quantum group $\U$ is an associative unital $\F$-algebra generated by $e$, $f$, $k^{\pm 1}$ subject to the relations
\begin{gather*}
kk^{-1}=k^{-1}k=1, \\
ke=q^2ek, \qquad \quad kf=q^{-2}fk, \\ ef-fe=\frac{k-k^{-1}}{q-q^{-1}}.
\end{gather*}
We review a classification of the finite-dimensional irreducible $\U$-modules from \cite[Theorem~2.6]{jantzen}.
For each $\e\in\{\pm 1\}$ there exists an $(n+1)$-dimensional irreducible $\U$-module $V_{n,\e}$ that has a basis with respect to which the matrices representing $e$, $f$, $k$ are superdiagonal, subdiagonal, diagonal respectively with
\begin{eqnarray*}
e_{i-1,i} &=& \e [n-i+1]_q \qquad \quad (1\leq i\leq n),\\
f_{i,i-1} &=& [i]_q \qquad \quad (1\leq i\leq n),\\
k_{ii} &=& \e q^{n-2i} \qquad \quad (0\leq i\leq n).
\end{eqnarray*}
Every $(n+1)$-dimensional irreducible $\U$-module is shown to be isomorphic to $V_{n,1}$ or $V_{n,-1}$. If ${\rm char}\,\F=2$ the $\U$-modules $V_{n,1}$ and $V_{n,-1}$ are isomorphic. In what follows, the parameter $\e$ will be called the {\it type} of $V_{n,\e}$ and the above basis of $V_{n,\e}$ will be said to be {\it canonical}.
Observe that the Casimir element
$$
\Lambda=ef+\frac{q^{-1}k+q k^{-1}}{(q-q^{-1})^2}
$$
acts on $V_{n,\e}$ as the scalar
$$
\e\frac{q^{n+1}+q^{-n-1}}{(q-q^{-1})^2}.
$$
It follows that

\begin{lem}\label{lem:Casimir}
Let $V$, $W$ denote two finite-dimensional irreducible $\U$-modules. If $\Lambda$ acts on $V$ and $W$ as the same scalar then $V$ and $W$ are isomorphic.
\end{lem}

The elements
\begin{equation*}
x=k^{-1}-q^{-1}(q-q^{-1})ek^{-1}, \qquad \quad
y^{\pm 1}=k^{\pm 1},\qquad \quad
z=k^{-1}+(q-q^{-1})f
\end{equation*}
form a set of generators of $\U$, which are called the {\it equitable generators} of $\U$ \cite[Definition~2.2]{equit2005}. %By (\ref{e:chev->equi})
The matrices representing $x$, $y$, $z$ with respect to the canonical basis of $V_{n,\e}$ are upper bidiagonal, diagonal, lower bidiagonal respectively with
\begin{align*}
x_{ii} &= \e q^{2i-n} \qquad \quad (0\leq i\leq n),\qquad \quad
x_{i-1,i}= q^{2i-n-1}(q^{i-n-1}-q^{n-i+1})
\qquad \quad (1\leq i\leq n),\\
y_{ii} &= \e q^{n-2i} \qquad \quad (0\leq i\leq n),\\
z_{ii} &= \e q^{2i-n} \qquad \quad (0\leq i\leq n),\qquad \quad
z_{i,i-1}= q^{i}-q^{-i} \qquad \quad (1\leq i\leq n).
\end{align*}
 Each of $x$, $y$, $z$ is diagonalizable on $V_{n,\e}$ with pairwise distinct eigenvalues $\e q^{n-2i}$ for all $0\leq i\leq n$. By \cite[Theorem~7.5]{equit2005} or applying Lemma~\ref{lem:Casimir} we see that

\begin{lem}\label{lem:xyz_cyclic}
For each finite-dimensional irreducible $\U$-module $V$, there exists an
invertible linear transformation $L$ on $V$ such that
\begin{equation*}
L^{-1} x L=y,\qquad \quad
L^{-1} y L=z, \qquad \quad
L^{-1} z L=x.
\end{equation*}
\end{lem}

We are ready to prove the main result of this section.

\begin{thm}\label{thm:U}
For each $(a,b,c)\in \T$ there are exactly
\begin{eqnarray*}
h=\left\{
\begin{array}{ll}
2 \qquad  &\hbox{if ${\rm char}\,\F\not=2$ and $a^2=b^2=c^2=-1$},\\
1 \qquad  &\hbox{otherwise}
\end{array}
\right.
\end{eqnarray*}
distinct $\U$-modules on $V_n(a,b,c)$ satisfying
\begin{eqnarray}
A &=& ax + a^{-1}y+ bc^{-1}\frac{xy-yx}{q-q^{-1}}, \label{UA}\\
B &=& by + b^{-1}z+ ca^{-1}\frac{yz-zy}{q-q^{-1}},\label{UB}\\
C &=& cz +c^{-1}x+ ab^{-1}\frac{zx-xz}{q-q^{-1}}. \label{UC}
\end{eqnarray}
Moreover these $\U$-modules are irreducible, one of which is of type $1$ and if $h=2$ then the other one is of type $-1$.
\end{thm}
\begin{proof}
The irreducibility of $V_n(a,b,c)$ forces that the $\U$-modules on $V_n(a,b,c)$ with (\ref{UA})--(\ref{UC}) are irreducible. We first show that there exists a unique irreducible $\U$-module on $V_n(a,b,c)$ of type $1$ satisfying (\ref{UA})--(\ref{UC}).

(Existence):
Let
$$
c_i=%\prod^{i-1}_{h=0}(q^{h+1}-q^{-h-1})(b^{-1}-a^{-1}c\,q^{n-2h-1})
\prod^{i}_{h=1}(q^{h}-q^{-h})(b^{-1}-q^{n-2h+1}a^{-1}c) \qquad \quad (0\leq i\leq n).
$$
Since $q$ is not a root of unity and $(a,b,c)\in \T$ the scalars $c_i$ are nonzero for all $0\leq i\leq n$. Define an $\U$-module $V$ on $V_n(a,b,c)$ such that the action of $x$, $y$, $z$ on $\{c_i^{-1}v_i^{(1,-1,\tau)}\}^n_{i=0}$ is the same as that on the canonical basis of $V_{n,1}$.
A direct calculation yields that $V$ satisfies (\ref{UA})--(\ref{UC}).
This shows the existence.

(Uniqueness):
Suppose that $V$ is any irreducible $\U$-module on $V_n(a,b,c)$ of type $1$ with (\ref{UA})--(\ref{UC}). To see the uniqueness, it is enough to show that
the $q^{n-2i}$-eigenspaces of $x$, $y$, $z$ on $V$ for $0\leq i\leq n$ are only determined by $V_n(a,b,c)$. Let $\{v_i\}^n_{i=0}$ denote the canonical basis of $V$. Then $\F v_i$ ($0\leq i\leq n$) is the $q^{n-2i}$-eigenspace of $y$ on $V$. The action of $A$, $B$, $C$ on $\{c_i v_i\}^n_{i=0}$ is the same as that on $\{v_i^{(1,-1,\tau)}\}^n_{i=0}$.
By Proposition~\ref{lem:24}, $\F v_i=\F v_i^{(1,-1,\tau)}$ for each $0\leq i\leq n$. Let $L$ denote the linear transformation on $V$ from Lemma~\ref{lem:xyz_cyclic}. Then $\F L v_i$ and $\F L^{-1} v_i$ ($0\leq i\leq n$) are the $q^{n-2i}$-eigenspaces of $x$, $z$ on $V$ respectively. By symmetry $\F L v_i=\F v_i^{(-1,1,\sigma)}$ and $\F L^{-1} v_i=\F v_i^{(-1,-1,\sigma\tau\sigma)}$ for each $0\leq i\leq n$. The uniqueness follows.

Now assume that ${\rm char}\,\F\not=2$.
On any $\U$-modules on $V_n(a,b,c)$ of type $-1$  with (\ref{UA})--(\ref{UC}), the traces of $A$, $B$, $C$ are equal to $-[n+1]_q(a+a^{-1})$, $-[n+1]_q(b+b^{-1})$, $-[n+1]_q(c+c^{-1})$ respectively. By Lemma~\ref{lem:chara}(ii) these $\U$-modules exist only if $a^2=b^2=c^2=-1$. Conversely, if $a^2=b^2=c^2=-1$ then a similar argument shows that there exists a unique $\U$-module of type $-1$ on $V_n(a,b,c)$ with (\ref{UA})--(\ref{UC}).
\end{proof}

\subsection{An application to the Racah coefficients of $\U$}
\label{s:Racah}

Recall that $\U$ has a coalgebra structure with comultiplication $\Delta:\U\to \U\otimes \U$ given by
\begin{eqnarray*}
\Delta(e) &=&e\otimes 1+k\otimes e,\\
\Delta(f) &=&f\otimes k^{-1}+1\otimes f,\\
\Delta(k) &=&k\otimes k.
\end{eqnarray*}
Denote by $\Delta_1=\Delta$ and recurrently define
$$
\Delta_{n+1}=(\Delta \otimes 1\otimes 1\otimes \cdots \otimes 1)\circ \Delta_{n}
\qquad \quad
\hbox{($1$ appearing $n$ times)}
$$
for all $n\in \N^*$. Therefore any $\U^{\otimes n}$-module ($n\geq 2$) can be treated as a $\U$-module by pulling back via $\Delta_{n-1}$. %

Roughly speaking, the Racah coefficients of $\U$ are used to describe the change of two natural bases of the $\U$-module $U\otimes V\otimes W$ for any finite-dimensional $\U$-modules $U$, $V$, $W$. By \cite[Theorem~2.9]{jantzen}
the finite-dimensional $\U$-modules are completely reducible provided that ${\rm char}\,\F\not=2$. For any $n\in \N$ and $\e\in \{\pm 1\}$ the irreducible $\U$-module $V_{n,\e}$ is isomorphic to $V_{0,\e}\otimes V_{n,1}$ and $V_{n,1}\otimes V_{0,\e}$. Thus it is enough to work on the $\U$-modules $V_{n,1}$ in general.
%, henceforth denoted by $V_n$ for simplicity.
Henceforth we denote by $V_n=V_{n,1}$ for simplicity.
The Clebsch-Gordan formula \cite[Section~5A.8]{jantzen} %states that $V_m\otimes V_n$ $(m,n\in \N)$ has a unique decomposition into the irreducible $\U$-modules which is isomorphic to
decomposes $V_m\otimes V_n$ $(m,n\in \N)$ into
$$
\bigoplus_{i=0}^{\min\{m,n\}}V_{m+n-2i}.
$$
%Furthermore the formula expresses each canonical basis vector of the direct summand $V_{m+n-2i}$ as a linear combination of the tensor products of the canonical basis vectors of $V_m$ and $V_n$.
As an application of the Clebsch-Gordan formula we obtain two bases of the $\U$-module
$$
V_m\otimes V_n\otimes V_p
\qquad \quad
(m,n,p\in \N).
$$
The first one comprises the canonical bases of the irreducible components of $V_{m+n-2i}\otimes V_p$ for all $0\leq i\leq \min\{m,n\}$. The second one is obtained by the same procedure beginning with the decomposition of $V_n\otimes V_p$. Denote by $\{u_i\}_{i=0}^N$, $\{v_i\}_{i=0}^N$ the two bases of $V_m\otimes V_n\otimes V_p$ respectively, where $N=(m+1)(n+1)(p+1)-1$. The {\it Racah coefficients} of $\U$ are defined to be the entries of the transition matrix from $\{u_i\}_{i=0}^N$ to $\{v_i\}_{i=0}^N$. The so-called Racah problem for $\U$ is to find an explicit expression for Racah coefficients of $\U$.

%A direct calculation yields that
%$$
%\Delta(\Lambda)=
%\Lambda\otimes k^{-1}
%+k\otimes \Lambda
%+e\otimes f
%+e^\dag \otimes f^\dag
%-\frac{q+q^{-1}}{(q-q^{-1})^2}\, k\otimes k^{-1}.
%$$

%$$
%(\Delta\otimes 1)(\Delta(\Lambda))=
%\Delta(\Lambda)\otimes k^{-1}
%+k\otimes k\otimes \Lambda
%+\Delta(e)\otimes f
%+\Delta(e)^\dag\otimes f^\dag
%-\frac{q+q^{-1}}{(q-q^{-1})^2} k\otimes k\otimes k^{-1}.
%$$
Inspired by the works \cite{gz92,gvz2013,gvz2015} on the Racah coefficients of $\mathfrak{sl}_{-1}(2)$, $\mathfrak{osp}_q(1|2)$ and $\mathfrak{su}_q(2)$, one may predict that

\begin{thm}\label{thm:sharp}
There exists a unique $\F$-algebra homomorphism $\triangle \to \U\otimes\U\otimes\U$ that sends
\begin{align*}
&\frac{A}
{(q-q^{-1})^2}
\quad \mapsto \quad
\Delta(\Lambda)\otimes 1,
\\
&\frac{B}
{(q-q^{-1})^2}
\quad \mapsto \quad
1\otimes \Delta(\Lambda),
\\
&\frac{\gamma}
{(q-q^{-1})^4}
\quad \mapsto \quad
\Lambda\otimes 1\otimes \Lambda
+(1\otimes \Lambda\otimes 1)
\cdot
\Delta_2(\Lambda).
\end{align*}
\end{thm}

\noindent
Pulling back via the homomorphism shown in Theorem~\ref{thm:sharp} the $\U\otimes\U\otimes\U$-module $V_m\otimes V_n\otimes V_p$ admits a $\triangle$-module structure. By Lemma~\ref{lem:Casimir} the bases $\{u_i\}_{i=0}^N$, $\{v_i\}_{i=0}^N$ are the eigenbases of $A$, $B$ on $V_m\otimes V_n\otimes V_p$ respectively. Applying the relations (\ref{a->ABc}), (\ref{b->ABc}) it follows that $A$, $B$ act on $\{v_i\}_{i=0}^N$, $\{u_i\}_{i=0}^N$ in tridiagonal fashions respectively. Moreover the $\triangle$-module $V_m\otimes V_n\otimes V_p$ is completely reducible and $A$, $B$ act on each irreducible component as a Leonard pair. All of the details will be covered in a future paper.

As a result, the Racah problem for $\U$ can be expanded to determine the transition matrices between the two eigenbases of an arbitrary Leonard pair. A solution in terms of hypergeometric series and their $q$-analogues can be found in \cite{lp2004}.

%\begin{lem}
%The $(n+1)\times (n+1)$ matrix $R(Q;X,Y,Z)$ with $(i,j)$-entry
%\begin{gather*}
%\sum_{k=0}^i\prod_{h=1}^k
%\frac
%{
%(\theta_i(Q;Y)-\theta_{h-1}(Q,Y))
%(\theta_j(Q;X)-\theta_{h-1}(Q,X))
%}
%{
%\phi_h(Q;X,Y,Z)
%}
%\end{gather*}
%for all $0\leq i,j\leq n$ is the unique matrix $R$ satisfying
%\begin{gather*}
%G(Q;X,Y,Z)\cdot R=R\cdot D(Q;X)
%\end{gather*}
%with $R_{i0}=1$ for all $0\leq i\leq n$.
%\end{lem}

%\bigskip

%\noindent Hau-wen Huang
%\hfil\break Mathematics Division
%\hfil\break National Center for Theoretical Sciences
%\hfil\break National Tsing-Hua University
%\hfil\break Hsinchu 30013, Taiwan, R.O.C.
%\hfil\break Email:  {\tt hauwenh@math.cts.nthu.edu.tw}

\end{document}